\documentclass[12pt, a4paper]{amsart}
\usepackage{amsmath}
\usepackage{geometry,amsthm,graphics,tabularx,amssymb,shapepar}
\usepackage[cmyk]{xcolor}
\usepackage{appendix}
\usepackage{amscd}
\usepackage[all]{xypic}
\usepackage{ulem}
\usepackage{bm}

\makeatletter
\newcommand*{\rom}[1]{\expandafter\@slowromancap\romannumeral #1@}
\makeatother

\newcommand{\BC}{{\mathbb {C}}}

\newcommand{\BN}{{\mathbb {N}}}

\newcommand{\BR}{{\mathbb {R}}}

\newcommand{\CC}{{\mathcal {C}}}

\newcommand{\CF}{{\mathcal {F}}}

\newcommand{\CL}{{\mathcal {L}}}

\newcommand{\CO}{{\mathcal {O}}}

\newcommand{\RC}{{\mathrm {C}}}
\newcommand{\RD}{{\mathrm {D}}}

\newcommand{\wt}{\widetilde}
\newcommand{\wh}{\widehat}

\newcommand{\oT}{\operatorname{T}}

\newcommand{\g}{\mathfrak g}

\newcommand{\Z}{\mathbb{Z}}
\newcommand{\C}{\mathbb{C}}
\newcommand{\R}{\mathbb R}

\newcommand{\cf}{\textit{cf}.~}
\newcommand{\la}{\langle}
\newcommand{\ra}{\rangle}

\newcommand{\be}{\begin {equation}}
\newcommand{\ee}{\end {equation}}
\newcommand{\bee}{\begin {equation*}}
\newcommand{\eee}{\end {equation*}}

\newcommand{\qaq}{\quad\textrm{and}\quad}

\renewcommand{\mid}{\,:\,}

\theoremstyle{Theorem}

\theoremstyle{Theorem}

\theoremstyle{Theorem}

\theoremstyle{Theorem}

\theoremstyle{Plain}

\theoremstyle{remark}

\theoremstyle{remark}

\theoremstyle{Definition}
\newtheorem{dfn}{Definition}[section]

\newtheorem{cord}[dfn]{Corollary}
\newtheorem{prpd}[dfn]{Proposition}
\newtheorem{thmd}[dfn]{Theorem}
\newtheorem{lemd}[dfn]{Lemma}
\newtheorem{remarkd}[dfn]{Remark}
\newtheorem{exampled}[dfn]{Example}
\numberwithin{equation}{section}

\begin{document}

\title[Poincar\'e's lemma]{Poincar\'e's lemma for formal manifolds}

\author[F. Chen]{Fulin Chen}
\address{School of Mathematical Sciences, Xiamen University,
 Xiamen, 361005, China} \email{chenf@xmu.edu.cn}

\author[B. Sun]{Binyong Sun}
\address{Institute for Advanced Study in Mathematics \& New Cornerstone Science Laboratory, Zhejiang University,  Hangzhou, 310058, China}
\email{sunbinyong@zju.edu.cn}

\author[C. Wang]{Chuyun Wang}
\address{Institute for Theoretical
 Sciences/ Institute of Natural Sciences, Westlake University/ Westlake Institute for Advanced Study,
 Hangzhou, 310030, China}
\email{wangchuyun@westlake.edu.cn}


\begin{abstract}
This is a paper in a series that studies 
smooth relative Lie algebra homologies and cohomologies
based on the theory of formal manifolds and formal Lie groups. 
In two previous papers, we develop the basic theory of formal manifolds, including generalizations of vector-valued distributions and generalized functions on smooth manifolds to the setting of formal manifolds. In this paper, we establish 
Poincar\'e's lemma for de Rham complexes with coefficients in formal functions, formal generalized functions, compactly supported formal densities, or compactly supported formal distributions. 
\end{abstract}

\maketitle

\tableofcontents

\section{Introduction}

This is a sequel to \cite{CSW1} and \cite{CSW2} in a series to study the theory of formal manifolds and formal Lie groups, with the ultimate goal of establishing a smooth relative Lie algebra  (co)homology theory.

\subsection{The motivation}

Motivated by recent developments in the theories of Lie group representations and automorphic forms, there is a growing interest in establishing a  (co)homology theory for smooth representations of general Lie pairs. Here, by a Lie pair, we mean a finite-dimensional complex Lie algebra and a  (not necessarily compact) Lie group, together with additional structures that satisfy the same compatibility conditions as that of the usual pair $(\g,K)$ for a reductive Lie group $G$ (see \cite[(1.64)]{KV} for example), where $\g$ is the complexified Lie algebra of $G$ and $K$ is a maximal compact subgroup of $G$.  We call this  (co)homology theory the
smooth relative Lie algebra (co)homology theory, which unifies the smooth (co)homology theory for Lie groups (see \cite{HM} and \cite{BW}) with the algebraic theory of relative Lie algebra (co)homologies (see \cite{KV}).

In the 1970s, Zuckerman proposed a homological construction of Harish-Chandra modules on certain relative Lie algebra (co)homology spaces (see \cite{KV}), which played an important role in the representation theory of Lie groups. 
We expect that there is an analytic analog of Zuckerman’s functor based on the smooth relative Lie algebra (co)homology theory so that it produces interesting smooth
representations of reductive Lie groups through cohomological induction in a similar way. 

Starting with \cite{CSW1}, we formulate and study a notion of what we call formal manifolds, inspired by the notion of formal schemes in algebraic geometry.
In a forthcoming paper, we will further explore formal
Lie groups, defined as group objects in the category of formal manifolds. Particularly, we
will prove therein that there is an equivalence of categories between the category of formal
Lie groups and that of general Lie pairs. 
Then we plan to establish the smooth relative Lie algebra (co)homology theory by studying (co)homologies of ``suitable" representations for formal Lie groups. 


In the  smooth (co)homology theory of Lie groups,  the standard (projective or injective) resolutions of a smooth representation 
can be constructed by using various functions spaces(see \cite{HM,BW,KS}). 
In \cite{CSW2} we introduce and study various functions spaces on formal manifolds, including generalizations of
vector-valued generalized functions and vector-valued distributions on smooth
manifolds to the setting of formal manifolds.
As a stone step towards the (co)homologies of representations for formal Lie groups, we need to construct standard resolutions of the trivial representation. 
 To achieve this goal, an important work is to define the de Rham complexes for formal manifolds with coefficients in various function spaces and prove the corresponding Poincar\'e's lemma. This is the main motivation of the present paper.

\subsection{Notations and conventions} Before presenting the main results of this paper, we provide some notations and conventions that are used throughout this paper. 
Let $N$ be a smooth manifold and $k\in \BN$. We denote by $N^{(k)}$ the locally ringed space $(N, \CO_N^{(k)})$ over $\mathrm{Spec}(\BC)$, where 
	\[
	\CO_N^{(k)}(U):=\RC^\infty(U)[[y_1, y_2, \dots, y_k]]\quad (U\ \text{is an open subset of}\ N)
	\]
denotes the algebra of formal power series with coefficients in the algebra  $\RC^\infty(U)$ of complex-valued smooth functions on $U$.
Recall from \cite[Definition 1.4]{CSW1} that a formal manifold  is a locally ringed space $(M, \CO)$ over $\mathrm{Spec}(\BC)$ such that
\begin{itemize}
    \item the topological space $M$ is paracompact and Hausdorff; and
    \item for every $a\in M$, there is an open  neighborhood $U$ of $a$ in $M$ and $n,k\in \BN$ such that $(U, \CO|_U)$ is isomorphic to $(\R^n)^{(k)}$ as locally ringed spaces over $\mathrm{Spec}(\BC)$.
\end{itemize}

Throughout this paper, let $(M,\CO)$ be a formal manifold. By abuse of notation,  we will often not distinguish it with the underlying topological space $M$.  For every $a\in M$, the numbers $n$ and $k$  are respectively called the dimension and the degree of $M$ at $a$,  denoted by $\dim_a(M)$ and $\deg_a(M)$. We denote by $\pi_0(M)$ the set of all connected components in $M$, which may or may not be countable. An element in $\CO(M)$ is called a formal function on $M$.

In this paper,  by an LCS, 
we mean a 
locally convex topological vector
space over $\C$, which may or may not be Hausdorff. 
However, a complete or quasi-complete LCS is always assumed to be Hausdorff. 
For two LCS $E_1$ and $E_2$, let $\CL(E_1,E_2)$ denote the space of all continuous linear maps from $E_1$ to $E_2$. When $E_2=\C$, we also set $E_1':=\CL(E_1,E_2)$. 
Unless otherwise mentioned,  $\CL(E_1,E_2)$ is equipped with the strong topology.

As stated in \cite{Gr}, there are three useful topological tensor products on LCS:
the inductive tensor product $\otimes_{\mathrm{i}}$, the projective tensor product $\otimes_\pi$, and the epsilon tensor product $\otimes_\varepsilon$.
We denote the quasi-completions and completions of these topological tensor products by
\[\widetilde\otimes_{\mathrm{i}},\,\, \widetilde\otimes_\pi,\,\, \widetilde\otimes_{\varepsilon}\quad \text{and}\quad 
\widehat\otimes_{\mathrm{i}}, \,\,\widehat\otimes_\pi,\,\, \widehat\otimes_{\varepsilon},\] respectively (see Appendix \ref{appendixC2} for more details).
Similar notations will be used for topological tensor products of 
topological cochain complexes (see Appendix \ref{appendixC2}).

Let $\CF$ be a sheaf of  $\CO$-modules and let $E$ be a quasi-complete LCS. Recall the (co)sheaves 
\[
\mathrm{D}_c^{\infty}(\CF),\ \mathrm{C}^{-\infty}(\CF;E),\ \CF_c,\ \mathrm{D}^{-\infty}(\CF;E),\ \mathrm{D}_c^{-\infty}(\CF;E)
\]
of $\CO$-modules, as well as the corresponding spaces 
\[
\mathrm{D}_c^{\infty}(M;\CF),\ \mathrm{C}^{-\infty}(M;\CF;E),\ \CF_c(M),\ \mathrm{D}^{-\infty}(M;\CF;E),\ \mathrm{D}_c^{-\infty}(M;\CF;E)
\] of global sections
introduced in \cite{CSW2}. 
When $E=\C$, set \[\mathrm{C}^{-\infty}(\CF):=\mathrm{C}^{-\infty}(\CF;E) \qaq \mathrm{C}^{-\infty}(M;\CF):=\mathrm{C}^{-\infty}(M;\CF;E),\] and we have the similar notations $\mathrm{D}^{-\infty}(\CF)$, 
$\mathrm{D}_c^{\infty}(\CF)$, $\mathrm{D}^{-\infty}(M;\CF)$ and $\mathrm{D}_c^{-\infty}(M;\CF)$. 
In particular, (see \cite[Theorem 1.1]{CSW2})
\begin{itemize}
    \item $\mathrm{D}_c^{\infty}(M;\CO)$ is the space of 
    compactly supported formal densities, and when $M=N^{(0)}$, it coincides with the space $\mathrm{D}_c^\infty(N)$ of compactly supported smooth densities on $N$;
    \item $\mathrm{C}^{-\infty}(M;\CO)$ is the space of formal generalized functions on $M$, and when $M=N^{(0)}$, it coincides with  the space $\mathrm{C}^{-\infty}(N)$ of generalized functions on $N$.
    \item $\mathcal{\CO}_c(M)$ is the space of compactly supported formal functions on $M$, and when $M=N^{(0)}$, it coincides with the space $\mathrm{C}_c^\infty(N)$ of compactly supported smooth functions on $N$;
    \item $\mathrm{D}^{-\infty}(M;\CO)$ is the space of formal distributions on $M$, and when $M=N^{(0)}$, it coincides with the space  $\mathrm{D}^{-\infty}(N)$ of distributions on $N$; and 
    \item  $\mathrm{D}_c^{-\infty}(M;\CO)$ is the space of compactly supported formal distributions on $M$, and when $M=N^{(0)}$, it coincides with  the space $\mathrm{D}_c^{-\infty}(N)$ of compactly supported distributions on $N$.
\end{itemize}
One may consult \cite{CSW2} for more details.

 \subsection{Main results}In Section \ref{sec:derham}, we  define the de Rham complex
\be\label{de1}
\Omega_\CO^\bullet: \ \cdots\rightarrow 0\rightarrow  0\rightarrow \Omega_\CO^0\rightarrow \Omega_\CO^1\rightarrow \Omega_\CO^2\rightarrow \cdots
\ee
for  $(M, \CO)$ with coefficients in formal functions, which is a complex of sheaves of $\CO$-modules.
By taking the transpose of $\Omega_\CO^\bullet$, we obtain a complex
\be \label{de2}
\mathrm{D}^{-\infty}_c(\Omega_\CO^{-\bullet}): \ \cdots\rightarrow  \mathrm{D}^{-\infty}_c(\Omega_\CO^2)\rightarrow \mathrm{D}^{-\infty}_c(\Omega_{\CO}^1)\rightarrow \mathrm{D}^{-\infty}_c(\Omega_\CO^0)\rightarrow 0\rightarrow \cdots
\ee
of cosheaves of $\CO$-modules.
We call $\mathrm{D}^{-\infty}_c(\Omega_\CO^{-\bullet})$ the de Rham complex
for $(M,\CO)$ with coefficients in compactly supported formal distributions.

As a subcomplex of $\mathrm{D}^{-\infty}_c(\Omega_\CO^{-\bullet})$, there is also a complex
\be\label{de3}
\mathrm{D}^{\infty}_c(\Omega_\CO^{-\bullet}): \ \cdots\rightarrow  \mathrm{D}^{\infty}_c(\Omega_\CO^2)\rightarrow \mathrm{D}^{\infty}_c(\Omega_{\CO}^1)\rightarrow \mathrm{D}^{\infty}_c(\Omega_\CO^0)\rightarrow 0\rightarrow \cdots
\ee
 of cosheaves of $\CO$-modules.
Similarly, by considering the  transpose of $\mathrm{D}^{\infty}_c(\Omega_\CO^{-\bullet})$, we have a complex
\be\label{de4}
\mathrm{C}^{-\infty}(\Omega_\CO^\bullet): \ \cdots\rightarrow 0\rightarrow  0\rightarrow \mathrm{C}^{-\infty}(\Omega_\CO^0)\rightarrow
\mathrm{C}^{-\infty}(\Omega_\CO^1)\rightarrow \mathrm{C}^{-\infty}(\Omega_\CO^2)\rightarrow \cdots
\ee
of sheaves of $\CO$-modules.
We call $\mathrm{D}^{\infty}_c(\Omega_\CO^{-\bullet})$ and $\mathrm{C}^{-\infty}(\Omega_\CO^\bullet)$ the de Rham complexes for
 $(M,\CO)$ with coefficients in compactly supported formal densities and formal generalized functions, respectively.

By taking the global sections in \eqref{de1}-\eqref{de4}, these de Rham complexes
produce four topological cochain complexes
\be\label{intro:de}\Omega^\bullet_\CO(M),\quad \mathrm{D}^{-\infty}_c(M;\Omega_\CO^{-\bullet}),\quad
\mathrm{D}^{\infty}_c(M;\Omega_\CO^{-\bullet})\quad\text{and}\quad \mathrm{C}^{-\infty}(M;\Omega_\CO^\bullet)\ee
 of reflexive complete LCS. 

Let $N_1$ and $N_2$ be two smooth manifolds. There are natural LCS identifications 
\begin{eqnarray}\label{eq:schwartzkernel1}
&&\RC^\infty(N_1\times N_2)=\RC^\infty(N_1)\widetilde\otimes_\pi\RC^\infty(N_2)=\RC^\infty(N_1)\widehat\otimes_\pi \RC^\infty(N_2),\\
\label{eq:schwartzkernel2}
&&\RD_c^{-\infty}(N_1\times N_2)=\RD_c^{-\infty}(N_1)\widetilde\otimes_{\mathrm{i}}\RD_c^{-\infty}(N_2)=\RD_c^{-\infty}(N_1)\widehat\otimes_{\mathrm{i}}\RD_c^{-\infty}(N_2),\\
\label{eq:schwartzkernel3}&&\RD^\infty_{c}(N_1\times N_2)=\RD^\infty_c(N_1)\widetilde\otimes_{\mathrm{i}}\RD^\infty_c(N_2)=\RD^\infty_c(N_1)\widehat\otimes_{\mathrm{i}} \RD^\infty_c(N_2),\\
\label{eq:schwartzkernel4}&&\RC^{-\infty}(N_1\times N_2)=\RC^{-\infty}(N_1)\widetilde\otimes_{\pi}\RC^{-\infty}(N_2)=\RC^{-\infty}(N_1)\widehat\otimes_{\pi}\RC^{-\infty}(N_2),
\end{eqnarray}
which are known as Schwartz kernel theorems for smooth manifolds (see \cite{Sc2} or \cite{Gr}). Recall from \cite[Theorem 6.18]{CSW1} that finite product exists in the category of formal manifolds.
In Section \ref{sec:tenderham}, we prove the following generalization of Schwartz kernel theorems in the setting of formal manifolds.

\begin{thmd}\label{Introthm:tensorofcomplex}
Let $(M_1,\CO_1)$ and $(M_2,\CO_2)$ be two 
formal manifolds and let
$(M_3,\CO_3)$
be the product of them. Then we have the following topological cochain complex identifications: 
\begin{eqnarray*}
\Omega^\bullet_{\CO_3}(M_3)&=&\Omega^\bullet_{\CO_1}(M_1)\widetilde\otimes_{\pi} \Omega^\bullet_{\CO_2}(M_2)\\
&=&\Omega^\bullet_{\CO_1}(M_1)\widehat\otimes_{\pi} \Omega^\bullet_{\CO_2}(M_2),\\
\mathrm{D}^{-\infty}_c(M_3;\Omega_{\CO_3}^{-\bullet})&=&\mathrm{D}^{-\infty}_c(M_1;\Omega_{\CO_1}^{-\bullet})\wt\otimes_{\mathrm 
 i}\mathrm{D}^{-\infty}_c(M_2;\Omega_{\CO_2}^{-\bullet})\\
&=& \mathrm{D}^{-\infty}_c(M_1;\Omega_{\CO_1}^{-\bullet})\widehat\otimes_{\mathrm{i}}  \mathrm{D}^{-\infty}_c(M_2;\Omega_{\CO_2}^{-\bullet}),\\
\mathrm{D}^{\infty}_c(M_3;\Omega_{\CO_3}^{-\bullet})&=&\mathrm{D}^{\infty}_c(M_1;\Omega_{\CO_1}^{-\bullet})
\widetilde\otimes_{\mathrm{i}} \mathrm{D}^{\infty}_c(M_2;\Omega_{\CO_2}^{-\bullet})\\
&=& \mathrm{D}^{\infty}_c(M_1;\Omega_{\CO_1}^{-\bullet})\widehat\otimes_{\mathrm{i}}  \mathrm{D}^{\infty}_c(M_2;\Omega_{\CO_2}^{-\bullet}),\\
\mathrm{C}^{-\infty}(M_3;\Omega_{\CO_3}^\bullet)&=&\mathrm{C}^{-\infty}(M_1;\Omega_{\CO_1}^\bullet)
\widetilde\otimes_\pi \mathrm{C}^{-\infty}(M_2;\Omega_{\CO_2}^\bullet)\\
&=& \mathrm{C}^{-\infty}(M_1;\Omega_{\CO_1}^\bullet)\widehat\otimes_\pi  \mathrm{C}^{-\infty}(M_2;\Omega_{\CO_2}^\bullet).
\end{eqnarray*}
\end{thmd}

 Let $\C_M$ denote the sheaf of locally constant $\BC$-valued functions on $M$.
As slight modifications of those complexes in \eqref{intro:de}, we have topological  cochain complexes
\begin{eqnarray}\label{extenddc1}
&&\cdots\rightarrow 0\rightarrow \BC_M(M)\rightarrow \Omega_\CO^0(M)\rightarrow \Omega_\CO^1(M)\rightarrow \cdots,\\
\label{extenddc2}&& \cdots \rightarrow \mathrm{D}^{-\infty}_c(M;\Omega_\CO^{1})\rightarrow \mathrm{D}^{-\infty}_c(M;\Omega_\CO^{0})\rightarrow
(\C_M(M))'\rightarrow 0\rightarrow \cdots,\\
\label{extenddc3}&& \cdots \rightarrow \mathrm{D}^{\infty}_c(M;\Omega_\CO^1)\rightarrow \mathrm{D}^{\infty}_c(M;\Omega_\CO^{0})\rightarrow
(\C_M(M))'\rightarrow 0\rightarrow \cdots,\\
\label{extenddc4}&& \cdots\rightarrow 0\rightarrow \BC_M(M)\rightarrow \mathrm{C}^{-\infty}(M;\Omega_{\CO}^0)\rightarrow  \mathrm{C}^{-\infty}(M;\Omega^1_{\CO})\rightarrow  \cdots
\end{eqnarray}
of reflexive complete   LCS.

Recall that a continuous linear map $\alpha: E_1\rightarrow E_2$ of complex topological vector spaces is said to be strong if there is a continuous linear map $\beta: E_2\rightarrow E_1$ such that $\alpha\circ\beta\circ \alpha=\alpha$.
Using Theorem \ref{Introthm:tensorofcomplex}, we prove in Section \ref{sec:poin} the following theorems, which generalize Poincar\'e's lemma for smooth manifolds. 

\begin{thmd}\label{pf}
	Suppose that $M=N^{(k)}$ for some contractible smooth manifold $N$ and some $k\in \BN$. Then the complexes \eqref{extenddc1}
 and \eqref{extenddc2} are strongly exact, namely, they are exact and all the arrows are strong continuous linear maps.
\end{thmd}

\begin{thmd}\label{pfc}
	Suppose that $M=(\R^n)^{(k)}$ for some $n,k\in \BN$. Then the complexes \eqref{extenddc3} and \eqref{extenddc4} are strongly exact.
\end{thmd}


\section{De Rham complexes for formal manifolds}\label{sec:derham}
In this section,  we introduce the de Rham complex for formal manifolds with coefficients in formal functions, compactly supported
formal distributions, compactly supported
formal densities, or formal generalized
functions. Throughout this section, let $E$ be a quasi-complete  LCS.
\subsection{The sheaves $\mathcal{D}\mathrm{er}(\CO)$ and $\CO_E$} 
In this subsection, we recall the notion of derivations of  $\mathcal{\CO}$ (see \cite[Section 3.2]{CSW1}), as well as the notion of  $E$-valued formal functions on $M$ (see  \cite[Section 6.2]{CSW1}).

For every $\C$-algebra $A$,  write $\mathrm{Der}(A)$ for the space of all derivations of $A$. Namely, 
\[\mathrm{Der}(A):=\{X\in \mathrm{Hom}_\C(A,A)\mid X(a_1a_2)=X(a_1)a_2+a_1X(a_2)\ \text{for all}\ a_1,a_2\in A\}.\]
By a derivation of $\CO$, we mean
 a family
\[
 D=\{
 (D_U: \CO(U)\rightarrow \CO(U))\in \mathrm{Der}(\CO(U))\}_{U\textrm{ is an open subset of }M}
\]
of derivations that commute with the restriction maps.
Write $\mathrm{Der}(\CO)$ for the space of all derivations of $\CO$.
By \cite[Proposition 3.8]{CSW1}, one concludes that the map 
\[
\mathrm{Der}(\CO)\rightarrow \mathrm{Der}(\CO(M)),\quad D\mapsto D_M
\]
is a bijection. In view of this, we will often not distinguish between the derivations $D$ and $D_M$.

With the obvious restriction maps, the assignment 
 \[
U\mapsto \mathrm{Der}(\CO|_U)\quad (\text{$U$ is an open subset of $M$})\]
forms a sheaf over $M$, to be denoted by $\mathcal{D}\mathrm{er}(\CO)$. 
Note that $\mathcal{D}\mathrm{er}(\CO)$ is naturally a sheaf of $\CO$-modules with the action given by  
\[ \CO(U)\times\mathrm{Der}(\CO|_U)\rightarrow \mathrm{Der}(\CO|_U),\quad \quad  (f,D)\mapsto f\circ D,\]
where $U$ is an open subset of $M$ and  $ f\circ D$ is defined as follows
\[
(f\circ D)_V: \ g\mapsto f|_V\cdot D_V(g)\quad (g\in \CO(V)\ \text{and $V$ is an open subset of $U$}).
\]

\begin{exampled}\label{ex:unders}Assume that $(M,\CO)=(N,\CO_N^{(k)})$ with $N$ a nonempty open submanifold of $\BR^n$ and $n,k\in \BN$. Let $x_1, x_2, \dots , x_n$ and $y_1,y_2,\dots,y_k$   denote the standard coordinate functions and formal variables of $M$, respectively. 
For every open subset $U$ of $N$, write  
\[
\partial_{x_1},\partial_{x_2},\dots,\partial_{x_n}\quad\text{and}\quad \partial_{y_1}, \partial_{y_2}, \dots, \partial_{y_k}
\]
for the first-order  partial  derivatives on  $\RC^\infty(U)[[y_1,y_2,\dots,y_k]]$ 
with respect to the above variables. 
By \cite[Corollary 3.12]{CSW1}, the sheaf  $\mathcal{D}\mathrm{er}(\CO)$ of  $\CO$-modules is free of rank $n+k$, with a set of free generators
 given by \[
\{\partial_{x_1}, \partial_{x_2}, \dots, \partial_{x_n},\, \partial_{y_1}, \partial_{y_2}, \dots, \partial_{y_k}\}.
\]
\end{exampled}

Equip  $\CO(M)$ with the smooth topology (see \cite[Definition 4.1]{CSW1}), then it becomes a product of nuclear Fr\'echet spaces (see \cite[Proposition 4.8]{CSW1}). Set 
\[
\CO_E(M):=\CO(M)\widetilde \otimes E.
\]
Here, we refer to  Appendix \ref{appendixC2} for the usual  notations 
\[
\otimes_{\pi},\quad\otimes_{\varepsilon},\quad\otimes_{\mathrm{i}}, \quad \widetilde\otimes_{\pi},\quad\widetilde\otimes_{\varepsilon},\quad\widetilde\otimes_{\mathrm{i}}, \quad\widetilde \otimes, \quad\widehat\otimes_{\pi},\quad\widehat\otimes_{\varepsilon},\quad\widehat\otimes_{\mathrm{i}},\quad\widehat \otimes,
\]
which are used to denote various topological tensor products, their quasi-completions, and their completions. 
By \cite[Theorem 6.11]{CSW1}, the assignment 
\[
 \CO_E:\ U\mapsto \CO_E(U)\quad (\text{$U$ is an open subset of $M$}),
\]
together with the obvious restriction maps, 
forms a sheaf of complex vector spaces over $M$. An element in $\CO_E(M)$ is called an $E$-valued formal function on $M$. Note that we have the LCS identification 
\[\CO_E(U)=\CO(U)\widehat \otimes E\] provided that $E$ is complete.
Additionally, $\CO_E$ is naturally a sheaf of $\CO$-modules with the action given by left multiplication.

\subsection{De Rham complexes I}\label{difff}
In this subsection,  we introduce the de Rham complex for formal manifolds with coefficients in $E$-valued formal functions.

\begin{dfn}
    Let $r\in \BN$. When $r>0$, an $E$-valued differential $r$-form on $M$ is 
 a  multiple 
 $\CO(M)$-linear map 
	\be\label{omeg8}
\omega: \ \underbrace{\mathrm{Der}(\CO)\times \mathrm{Der}(\CO)\times \cdots \times \mathrm{Der}(\CO)}_{r} \rightarrow  \CO_E(M)=\CO(M)\wt\otimes E
	\ee
that is  alternating, namely,  
\[
       \omega(X_1, X_2, \dots, X_r)=0
       \]
     for all $X_1, X_2, \dots, X_r\in \mathrm{Der}(\CO)$ such that $X_i=X_j$ for some $1\leq i<j\leq r$.
      When $r=0$,  an $E$-valued differential $r$-form on $M$ is an element of $\CO_E(M)$.
\end{dfn}
   Write $\Omega^r_\CO(M;E)$ for the space of all $E$-valued differential  $r$-forms on $M$. The space $\Omega^r_\CO(M;E)$ is naturally an $\CO(M)$-module.

The following result is an easy consequence of  \cite[Proposition 3.8]{CSW1}.

\begin{prpd}\label{prop:eqdefnattra} Let $r$ be a positive integer. Then for every $E$-valued differential $r$-form $\omega$ on $M$, there exists a unique natural transform  (between sheaves of sets) \be\label{eq:nattra}
 \widetilde \omega=\{\widetilde \omega_U\}_{U\textrm{ is an open subset of }M}: \  \underbrace{\mathcal D\mathrm{er}(\CO)\times \mathcal D\mathrm{er}(\CO)\times \cdots \times \mathcal D\mathrm{er}(\CO)}_{r} \rightarrow  \CO_E
\ee
satisfying the following conditions:
\begin{itemize}
    \item  for every open subset $U$ of $M$, the map 
    \[
    \widetilde \omega_U:\  \underbrace{\mathrm{Der}(\CO|_U)\times \mathrm{Der}(\CO|_U)\times \cdots \times \mathrm{Der}(\CO|_U)}_{r} \rightarrow  \CO_E(U)
    \]
    is  multiple $\CO(U)$-linear and  alternating;
    \item $\widetilde \omega_M=\omega$.
\end{itemize}
\end{prpd}

For every $r\in \BN$, Proposition \ref{prop:eqdefnattra} implies that
the assignment
\[
  U\mapsto \Omega^r_{\CO}(U;E):=\Omega^r_{\CO|_U}(U;E)\quad(\text{$U$ is an open subset of $M$}),
\]
together with the restriction maps
\[\Omega^r_{\CO}(U;E)\rightarrow \Omega^r_{\CO}(V;E),\quad \omega\mapsto \omega|_{V}\quad(\text{$V\subset U$ are open subsets of $M$}),\]
forms a sheaf of $\CO$-modules, 
which we denote by $\Omega^r_{\CO}(E)$. Here, when $r>0$, $\omega|_{V}$ is defined to be the element $\wt\omega_V$  as in \eqref{eq:nattra}.
 Set \[ \Omega_{\CO}(M;E):=\bigoplus_{r\in \BN}\Omega^r_{\CO}(M;E)\quad\text{and}\quad
\Omega_{\CO}(E):=\bigoplus_{r\in \BN}\Omega^r_{\CO}(E).\]
When $E=\C$, we will simply denote  $\Omega^r_{\CO}(M;E)$, $\Omega^r_{\CO}(E)$, $\Omega_{\CO}(M;E)$, and $\Omega_{\CO}(E)$  by  $\Omega^r_{\CO}(M)$,
$\Omega^r_{\CO}$, $\Omega_{\CO}(M)$, and $\Omega_{\CO}$, respectively.

The space $\Omega^0_{\CO}(M;E)=\CO(M)\wt\otimes E$ is equipped with the projective tensor product topology. When $r>0$,
equip $\Omega^r_{\CO}(M;E)$ with the point-wise convergence topology. Namely, a net $\{\omega_i\}_{i\in I}$ in  $\Omega^r_{\CO}(M;E)$
  converges to $0$ if and only if
\[\text{the net $\{\omega_{i}(X_1,X_2,\dots,X_r)\}_{i\in I}$ converges to $0$ in $\CO_E(M)$}
	\]
for all $(X_1,X_2,\dots,X_r)\in \mathrm{Der}(\CO)^r$.
Under this topology, $\Omega^r_{\CO}(M;E)$ is naturally a Hausdorff LCS. 

\begin{lemd}\label{1con}
	Let $r\in \BN$ and let $U$ be an open subset of $M$. Then the restriction map \[\Omega^r_{\CO}(M;E)\rightarrow \Omega^r_{\CO}(U;E)\]
is continuous.
 	\end{lemd}
\begin{proof} The case when $r=0$ is proved in \cite[Lemma 6.13]{CSW1}. Now we assume that $r>0$.
	For each $a\in U$, choose open neighborhoods $V_a$ and $U_a$ of $a$
 such that \[ V_a\subset \overline{V_a}\subset U_a \subset U\qquad (\text{$\overline{V_a}$ is the closure of $V_a$ in $U$}).\]
	Take a  formal function $g^a$ on $M$ such that (see \cite[Corollary 2.4]{CSW1})
 \[g^a|_{\overline{V_a}}=1\quad \txt{and} \quad g^a|_{M\backslash U_a}=0.\]
Let  $\{\omega_{i}\}_{i\in I}$ be a net in $\Omega^r_{\CO}(M;E)$ that  converges to $0$.
 Then  for each $(X_1,X_2,\dots,X_r)\in \mathrm{Der}(\CO|_U)^r$, we obtain that 
 \begin{eqnarray*} 
&&	\lim_{i\in I} (\left(\omega_i|_{U}(X_1,X_2,\dots,X_r)\right)|_{V_a}) \\
&=&	\lim_{i\in I} (\omega_{i}|_{V_a}(X_1|_{V_a},X_2|_{V_a},\dots,X_r|_{V_a}))\\
	&=&\lim_{i\in I} (\left(\omega_i( g^aX_1,g^aX_2,\dots,g^aX_r)\right)|_{V_a})\\&=& (\lim_{i\in I} \left(\omega_i( g^aX_1,g^aX_2,\dots,g^aX_r)\right))|_{V_a}=0. 
\end{eqnarray*}
Here $g^aX_i\in \mathrm{Der}(\CO)$ is the extension by zero of $g^a|_UX_i$ for $1\leq i\leq r$. By \cite[Lemma  6.14]{CSW1}, it follows that $\lim_{i\in I} \left(\omega_i|_{U}(X_1,X_2,\dots,X_r)\right)=0$,
and the lemma then follows. 
\end{proof}

\begin{lemd}\label{emd}
Let $r\in\BN$ and let $\{U_{\gamma}\}_{\gamma\in \Gamma}$ be an open cover of $M$. Then the linear map \be \label{embd} \Omega^r_{\CO}(M;E)\rightarrow \prod_{\gamma\in \Gamma}\Omega^r_{\CO}(U_{\gamma};E)\ee  is a closed topological embedding.
\end{lemd}
\begin{proof} The case when $r=0$ is proved in \cite[Lemma 6.14]{CSW1}. Now we assume that $r>0$. Lemma \ref{1con} implies that the map \eqref{embd} is continuous,
while the sheaf property of $\Omega^r_{\CO}(E)$ implies that the map \eqref{embd} is injective and has closed image.
Thus, it remains to prove that, if $\{\omega_{i}\}_{i\in I}$ is a net in $\Omega^r_{\CO}(M;E)$ whose image under the map \eqref{embd} converges to zero,
then $\{\omega_i\}_{i\in I}$ converges to zero.

For every $\gamma\in \Gamma$ and $(X_1,X_2,\dots, X_r)\in\mathrm{Der}(\CO)^r$, we have that
\begin{eqnarray*}&&
\lim_{i\in I}(\left(\omega_{i}(X_1,X_2,\dots,X_r)\right)|_{U_\gamma})\\
&=&\lim_{i\in I} (\omega_i|_{U_\gamma}(X_1|_{ U_\gamma},X_2|_{U_\gamma},\dots,X_r|_{ U_\gamma})) =0.\end{eqnarray*}
This, together with  \cite[Lemma 6.14]{CSW1},  
shows that  \[\lim_{i\in I} (\omega_{i}(X_1,X_2,\dots,X_r)) =0,\] as desired.
\end{proof}

As usual, for $r,s\in \mathbb N$, we define the wedge product
\be\label{eq:wedge}
\wedge:\ 
\Omega^r_\CO(M)\times \Omega^s_{\CO}(M;E)\rightarrow \Omega^{r+s}_{\CO}(M;E),\quad (\omega_1,\omega_2)
\mapsto \omega_1\wedge\omega_2
\ee
 by setting
\begin{eqnarray*}\label{eq:wedgeofw}
  &&(\omega_1\wedge \omega_2)(X_1,X_2, \dots, X_{r+s})\\
 & :=&\frac{1}{r! s!}\sum_\sigma \mathrm{sign} (\sigma)\omega_{1}(X_{\sigma(1)}, X_{\sigma(2)}, \dots, X_{\sigma(r)})\cdot  \omega_{2}(X_{\sigma(r+1)}, X_{\sigma(r+2)}, \dots, X_{\sigma(r+s)}),\notag
\end{eqnarray*}
where  $(X_1,X_2,\dots,X_{r+s})\in \mathrm{Der}(\CO)^{r+s}$, and
$\sigma$ runs over all the permutations of $\{1,2,\dots, r+s\}$.
By using the fact that $\CO(M)$ is a topological $\BC$-algebra (see \cite[Proposition 4.8]{CSW1}), the module structure map \[\CO(M)\times \CO_E(M)\rightarrow \CO_E(M),\quad (f,g)\rightarrow fg\] is bilinear and continuous. Then the wedge product  \eqref{eq:wedge} is bilinear and continuous as well. Endow \[\Omega_{\CO}(M;E)=\bigoplus_{r\in\BN}\Omega_{\CO}^r(M;E)\] with the direct sum topology.
Then   $\Omega_{\CO}(M;E)$ is naturally an  $\Omega_\CO(M)$-module such that the module structure map \[ \begin{array}{rcl}
 \Omega_{\CO}(M) \times  \Omega_{\CO}(M;E)  & \rightarrow & \Omega_{\CO}(M;E), \\
    (\omega_1,\omega_2 ) &\mapsto & \omega_1\wedge \omega_2
\end{array}\] is bilinear and continuous. Moreover, $(\Omega_{\CO}(M),\wedge)$ is a graded topological $\BC$-algebra, which is super-commutative in the sense that
\[
  \omega_1\wedge \omega_2=(-1)^{rs} \omega_2\wedge \omega_1\quad (\omega_1\in\Omega^r_\CO(M), \omega_2\in\Omega^s_\CO(M)).
\]

For each $\omega\in \Omega^r_{\CO}(M;E)$, define its differential $d\omega\in \Omega^{r+1}_{\CO}(M;E)$
by setting 
\be\label{eq:coboundary} \begin{array}{rl}
& (d\omega)(X_0, X_1, \dots, X_r)\\
   :=&\sum\limits_{0\le i\le r} (-1)^i (X_i\otimes \mathrm{id}_E)(\omega(X_0, X_1, \dots,\hat{X_i}, \dots, X_r))\\
  +&\sum\limits_{0\le i<j\le r} (-1)^{i+j}(\omega([X_i, X_j], X_0, X_1, \dots,\hat{X_i}, \dots,\hat{X_j}, \dots , X_r)),\end{array}
\ee
where  $(X_0,X_1,\dots,X_r)\in \mathrm{Der}(\CO)^{r+1}$,  $\mathrm{id}_E$ denotes the identity map on $E$, and “$
\hat{\phantom X}$” indicates the omission of a term. 
As usual, we have the following  straightforward results (\cf \cite[Theorem 2.20]{War}):
\begin{itemize}
\item the linear map $d:\Omega^r_{\CO}(M;E)\rightarrow\Omega^{r+1}_{\CO}(M;E)$ is continuous;
\item $d\circ d:\Omega^r_{\CO}(M;E)\rightarrow \Omega^{r+2}_{\CO}(M;E)$ is the zero map; and
\item for $\omega_1\in \Omega^r_{\CO}(M)$ and $\omega_2\in \Omega^s_{\CO}(M;E)$, \be \label{eq:dwedge} d(\omega_1\wedge \omega_2)=(d\omega_1)\wedge \omega_2+(-1)^r \omega_1\wedge (d\omega_2).\ee
\end{itemize}
These  imply that $(\Omega_{\CO}(M),\wedge,d)$ is a differential  graded topological $\BC$-algebra.
Meanwhile, it is routine to check that the family
 \be \label{eq:dof1} \{d:\ \Omega^r_{\CO}(U;E)\rightarrow \Omega^{r+1}_{\CO}(U;E)\} _{\text{$U$ is an open subset of $M$}}\ee 
is  a $\BC_M$-homomorphism between the sheaves $\Omega^r_{\CO}(E)$ and $\Omega^{r+1}_{\CO}(E)$.

For $n,r\in \BN$, put
\begin{eqnarray*}
	\Lambda_n^r:=\{( i_1,i_2,\dots,i_r)\in\BN^r\mid 1\leq  i_1<i_2< \cdots <i_r\leq n\},
	\end{eqnarray*}
which is an empty set when $r>n$. Furthermore, for $k\in \BN$, set
 \be\label{lambdankr}
\Lambda_{n,k}^r:=\bigsqcup_{0\le s\le r}\Lambda_n^s\times \Lambda_k^{r-s}.
\ee

\begin{exampled}\label{ex:DF}
In the setting of Example \ref{ex:unders}, 
for $1\le i\le n$ and $1\le j\le k$,  the differential $1$-forms $dx_i$  and $dy_j$  are the $\CO(M)$-linear maps  from $\mathrm{Der}(\CO)$ to $\CO(M)$ respectively determined by 
	\begin{eqnarray*}
 \partial_{x_{i'}}\mapsto \delta_{i,i'},\quad
\partial_{y_{j'}}\mapsto0\quad\text{and}\quad
\partial_{x_{i'}}\mapsto 0,\quad
		\partial_{y_{j'}}\mapsto \delta_{j,j'}
	\end{eqnarray*}
 for $i'=1,2,\dots,n$ and $j'=1,2,\dots,k$. Here $\delta_{i,i'}$ and $\delta_{j,j'}$ are the values of Kronecker delta function.
Using this, for every $f\in\Omega^0_{\CO}(M;E)=\CO_E(M)$, we have that
\[
df=\sum_{i=1}^n dx_i\wedge \partial_{x_i}(f) + \sum_{j=1}^k dy_j\wedge  \partial_{y_j}(f).\]
 If $r> n+k$, then $\Omega^r_{\CO}(M;E)=0$. For $0\le r\leq n+k $,
 every element in $\Omega^r_{\CO}(M;E)$ is uniquely of  the form
\[\sum_{(I,J)\in \Lambda_{n,k}^{r}} f_{I,J} dx_I dy_J,\] 
where  $f_{I,J}\in \CO_E(M)$ and for $(I,J)=( (i_1,i_2,\dots,i_s),(j_1,j_2,\dots,j_{r-s}))\in \Lambda_{n,k}^r$,
\[f_{I,J} dx_I dy_J:=dx_{i_1}\wedge dx_{i_2}\wedge\cdots \wedge dx_{i_s}\wedge dy_{j_1}\wedge dy_{j_2}\wedge \cdots \wedge dy_{j_{r-s}}\wedge f_{I,J}.\]
Thus, we have the following topological linear isomorphism:
 \be\label{eq:DFid} \begin{array}{rcl}
 \Omega^r_{\CO}(M;E)&\rightarrow&\prod\limits_{(I,J)\in \Lambda_{n,k}^r}\CO_E(M),\\ \sum f_{I,J} dx_I dy_J&\mapsto& (f_{I,J}).
 \end{array}\ee 
This, together with  \cite[Lemma 6.6]{CSW1}, implies  that \be\label{eq:Omega=OwtE}\Omega^r_{\CO}(M;E)=\Omega^r_{\CO}(M)\wt\otimes E\ee as LCS.
\end{exampled}

\begin{remarkd}\label{rem:DFtop}
For every $r\in \BN$, Example \ref{ex:DF} implies that the sheaf $\Omega^r_\CO$ of $\CO$-modules is locally free of finite rank. 
Furthermore, one concludes from \cite[Lemma 4.6 and Theorem 6.11]{CSW1}, Lemma \ref{emd}, and \eqref{eq:DFid} that
\begin{itemize}
    \item the topology on $\Omega^r_{\CO}(M)$
coincides with its smooth topology; 
\item  $\Omega^r_\CO(M;E)$ is  quasi-complete; and 
\item $\Omega^r_\CO(M;E)$ is complete provided that $E$ is complete.
\end{itemize}

\end{remarkd}

As in the case of smooth manifolds, we introduce the following definition.
\begin{dfn}
The cochain complex
\[
\Omega_{\CO}^\bullet(E):\quad \cdots\rightarrow 0\rightarrow 0\rightarrow\Omega_{\CO}^0(E)\xrightarrow{d} \Omega_{\CO}^1(E)\xrightarrow{d} \Omega^2_{\CO}(E)\xrightarrow{d} \cdots.
\]
of sheaves on $M$ is called the de Rham complex for $(M,\CO)$ with coefficients in $E$-valued formal functions.
\end{dfn}
By taking the global sections in $\Omega_{\CO}^\bullet(E)$, we have a topological cochain complex (see Definition \ref{de:topcomplex}) as follows:
\be \label{eq:Omegabullet}\Omega_\CO^{\bullet}(M;E):\ \cdots\rightarrow  0 \rightarrow\Omega^0_{\CO}(M;E)\xrightarrow{d} \Omega^1_{\CO}(M;E)\xrightarrow{d} \Omega^2_{\CO}(M;E)\xrightarrow{d} \cdots. \ee
For simplicity, we also set
\[\Omega_\CO^\bullet:=\Omega_\CO^\bullet(\C)\quad\text{and}\quad \Omega_\CO^{\bullet}(M):=\Omega_\CO^{\bullet}(M;\C).\]

Recall from \cite[Section 2.1]{CSW1} that  an open subset $U$ of $M$ is called a chart if  there exist $n,k\in \BN$ and an open
submanifold $N$ of $\R^n$ such that $(U,\CO|_U)\cong (N,\CO_N^{(k)})$ as $\C$-locally ringed spaces. 
An open cover of $M$ consisting of charts will be called an atlas of $M$.

We also recall some notions related to the quasi-completeness of LCS (see \cite{Sc} or  \cite[Section 6.1]{CSW1}). Let $A$ be a subset of an LCS $F$. Then
\begin{itemize}
    \item $A$ is called quasi-closed if it contains all limit points of the bounded subsets in it;
\item  the quasi-closure of $A$ is defined as the intersection of all the quasi-closed subsets of $F$  that contain $A$; and 

\item $A$ is called strictly dense in $F$ if its quasi-closure is $F$.
\end{itemize}
For every $r\in\BN$, the following result implies that
  the assignment
\[U\mapsto \Omega^r_\CO(U)\widetilde\otimes E\quad (\text{$U$ is an open subset of $M$})\] forms a sheaf of $\CO$-modules over $M$.
\begin{prpd}\label{prop:EvaluedDF} For each $r\in \BN$, the natural linear map
\[\Omega^r_\CO(M)\otimes_{\pi} E\rightarrow \Omega^r_\CO(M;E)\]
is continuous, and induces a topological linear isomorphism
\[\Omega^r_\CO(M)\widetilde\otimes E\rightarrow \Omega^r_\CO(M;E).\]
Furthermore, if $E$ is complete, then 
\[\Omega^r_{\CO}(M)\widetilde\otimes E=\Omega^r_\CO(M;E)=\Omega^r_\CO(M)\widehat\otimes E\] as LCS.
\end{prpd}
\begin{proof} The case when $r=0$ is proved in \cite[Theorem 6.11]{CSW1}. Now we assume that $r>0$. Take an atlas  $\{U_\gamma\}_{\gamma\in \Gamma}$ of $M$, and
consider the commutative diagram
\[
 \begin{CD}
              \Omega^r_\CO(M)\otimes_{\pi} E@>  >>  \Omega^r_\CO(M;E) \\
            @V  VV           @V V V\\
            \left(\prod_{\gamma\in \Gamma} \Omega^r_\CO(U_\gamma)\right)\widetilde\otimes E @ >> >
            \prod_{\gamma\in \Gamma} \Omega^r_\CO(U_\gamma;E).\\
  \end{CD}
\]
Lemma \ref{emd} and the fact 
that the epsilon tensor product of two linear topological embeddings of Hausdorff LCS is a linear topological embedding (see the proof of \cite[Proposition 43.7]{Tr}) implies that the left and right vertical arrows are both linear topological embeddings. 
\cite[Lemma 6.6]{CSW1} and \eqref{eq:Omega=OwtE}   imply that the bottom horizontal arrow is a topological linear isomorphism.
Thus the top horizontal arrow is a linear topological embedding.


By Remark \ref{rem:DFtop} and \cite[Lemma 6.5]{CSW1}, we only need to show that $\Omega^r_\CO(M)\otimes E$ is strictly dense in $\Omega^r_\CO(M;E)$.
Pick a partition of unity $\{f_\gamma\}_{\gamma\in \Gamma}$  on $M$ subordinate to $\{U_\gamma\}_{\gamma\in \Gamma}$ (see \cite[Proposition 2.3]{CSW1}).
By applying Lemma \ref{emd} to the open cover $\{U_\gamma, M\setminus \mathrm{supp}\,f_\gamma\}$ of $M$,  there is a well-defined continuous linear map
\[m_{f_\gamma}: \Omega^r_\CO( U_\gamma;E)\rightarrow  \Omega^r_\CO(M;E),
\quad \omega'\mapsto f_\gamma\omega',\]
where $f_\gamma\omega'$ is the extension by zero of $f_\gamma|_{U_\gamma}\omega'$.

Let $A$ be a quasi-closed subset of $\Omega^r_\CO(M;E)$ containing $\Omega^r_\CO(M)\otimes E$, and let $\omega\in \Omega^r_\CO(M;E)$.
Since the inverse image of a quasi-closed set under a continuous linear map is still quasi-closed  (see \cite[Page 92, $2^\circ$)]{Sc}),
it follows that $m_{f_\gamma}^{-1}(A)=\Omega^r_\CO(U_{\gamma};E)$ and $f_{\gamma}(\omega|_{U_{\gamma}})\in A$.
Hence 
\[B:=\left\{\sum_{\gamma\in \Gamma_0}f_\gamma(\omega|_{U_\gamma}) \mid \Gamma_0\ \text{is a finite subset of}\ \Gamma\right\}\]
is a bounded subset in $A$.
This  implies that $\omega=\sum_{\gamma\in \Gamma}f_\gamma(\omega|_{U_\gamma}) $ lies in the closure of $B$ and hence in $A$, as required.
\end{proof}

By using Proposition \ref{prop:EvaluedDF} and \eqref{eq:coboundary}, it is easy to check that 
 \be \label{EEE}
 \Omega_\CO^{\bullet}(M;E)=\Omega_{\CO}^{\bullet}(M)\widetilde\otimes\, \iota^{\bullet}(E)
 \ee
 as complexes of LCS. 
 Here the topological cochain complex $\iota^{\bullet}(E)$ is defined as in \eqref{eq:iotaE}. 
We refer to  Appendix \ref{appendixC2} for the notions of various tensor products of topological cochain complexes. 
	
\begin{prpd}\label{pro}
Let $\varphi=(\overline\varphi, \varphi^*): (M, \CO)\rightarrow (M',\CO')$ be a morphism of formal manifolds.
Then  there is a unique homomorphism
\be \label{eq:natural} \varphi^{\natural}:\ \overline\varphi^{-1}\Omega_{\CO'}\rightarrow \Omega_{\CO}\ee  of sheaves of graded algebras such that
\begin{itemize}
\item   $\varphi^{\natural}|_{\overline\varphi^{-1}\Omega^0_{\CO'}}=\varphi^*:\ \overline\varphi^{-1}\Omega^0_{\CO'}=\overline\varphi^{-1}\CO'\rightarrow \Omega^0_{\CO}=\CO$; and
 \item $\varphi^{\natural}\circ d=d\circ\varphi^{\natural}$.
\end{itemize}
Furthermore, the homomorphism
\[\varphi^{\natural}:\Omega_{\CO'}(M')\rightarrow \Omega_\CO(M)\]
 between the spaces of global sections is continuous.
\end{prpd}

\begin{proof}Without loss of generality, we assume that $M'\neq\emptyset$.
Assume first that  $M'=N^{(l)}$,
where $N$ is a nonempty open submanifold of $\R^m$,  and $m,l\in \BN$.
Let $\varphi^\natural$ be a homomorphism that fulfills the requirements in the proposition. Let $x_1, x_2, \dots , x_m$ and $y_1,y_2,\dots,y_l$ denote the standard coordinate functions and formal variables of $M'$, respectively. Then we have that
\[\varphi^{\natural}(dx_i)=d(\varphi^{\natural}(x_i))=d(\varphi^{*}(x_i))\quad  \textrm{and}\quad \varphi^{\natural}(dy_j)=d(\varphi^{*}(y_j)),\]
for $i=1,2,\dots,m$ and $j=1,2,\dots,l$.
Thus, for every $f\in \CO'(M')$ and \[(I,J)=((i_1,i_2,\dots,i_s),(j_1,j_2,\dots,j_{r-s}))\in \Lambda_{m,l}^{r}\quad \quad \text{($r\in\BN$),}\] the following equality holds (see Example \ref{ex:DF}):
\begin{eqnarray}\label{eq:defvarnatural}
&&\varphi^\natural(f dx_I dy_J)\\
&=& d(\varphi^*(x_{i_1}))\wedge \cdots \wedge
d(\varphi^*(x_{i_s}))\wedge d(\varphi^*(y_{j_1}))\wedge \cdots \wedge d(\varphi^*(y_{j_{r-s}}))\wedge \varphi^*(f).\notag
\end{eqnarray}
This proves the uniqueness of  $\varphi^\natural$.
On the other hand, the desired homomorphism $\varphi^\natural$ is defined by using \eqref{eq:defvarnatural}, which proves the existence.

For the general case, 
let $\varphi^{\natural}$ be a homomorphism that fulfills the requirements of this proposition. Then for each chart $U'\subset M'$, the homomorphism \[\varphi^\natural|_U: \overline\varphi^{-1}\Omega_{\CO'|_{U'}}\rightarrow \Omega_{\CO|_{U}}\quad \text(U:=\overline\varphi^{-1}(U'))\] satisfies the conditions in this proposition. The above argument implies that $\varphi^\natural|_U$ is uniquely determined, and consequently, so is $\varphi^\natural$.

To prove the existence, take an atlas  $\{U'_\gamma\}_{\gamma\in \Gamma}$ of $M'$ and put
$U_\gamma:=\overline\varphi^{-1}(U'_\gamma)$.
Then for each $\gamma$, there is a unique homomorphism
\[\varphi^\natural_\gamma: \overline\varphi^{-1}\Omega_{\CO'|_{U_\gamma'}}\rightarrow \Omega_{\CO|_{U_\gamma}}\]
satisfies the conditions in the proposition.
The uniqueness
allows us to glue $\{\varphi^{\natural}_{\gamma}\}_{\gamma\in \Gamma}$ together.
Then we obtain  a homomorphism
\[\varphi^\natural:\overline\varphi^{-1}\Omega_{\CO'}\rightarrow \Omega_{\CO}, \]
which satisfies the requirements in the proposition.
This completes the proof of the first assertion in the proposition.

Now we turn to prove that the map $\varphi^\natural$ is continuous.
For each $r\in \BN$ and $\gamma\in \Gamma$, one  deduces from
\eqref{eq:DFid},  \eqref{eq:defvarnatural}, and \cite[Theorem 4.15]{CSW1} that the map
\[
\varphi^{\natural}_{U'_\gamma}:\ \Omega^r_{\CO'}(U_\gamma')\rightarrow \Omega^r_\CO(U_\gamma)
\]
is continuous.
Then the continuity of $\varphi^\natural$ follows from Lemma \ref{emd} and the  commutative diagram
\[
	\begin{CD}
	\Omega^r_{\CO'}(M')	@>{\varphi^\natural} >>  \Omega^r_\CO(M) \\
		@V  VV           @V V V\\
	\prod_{\gamma\in\Gamma}\Omega^r_{\CO'}(U'_\gamma) @ > {\prod\varphi^\natural_{U'_{\gamma}}}> > \prod_{\gamma\in \Gamma}
\Omega^r_\CO(U_\gamma).\\
	\end{CD}
	\]
\end{proof}

\subsection{De Rham complexes II}
In this subsection, we introduce the other three de Rham complexes associated to $(M,\CO)$.

By transposing $\Omega_\CO^\bullet(M)$, we get  a  topological  cochain complex (see  Appendix \ref{appendixC3})
\[
\mathrm D^{-\infty}_c(M;\Omega_{\CO}^{-{\bullet}}):={}^t\Omega_\CO^\bullet(M)\] given by 
\[\cdots\rightarrow  \mathrm D^{-\infty}_c(M;\Omega_{\CO}^{2})\xrightarrow{d}
	\mathrm D^{-\infty}_c(M;\Omega_{\CO}^{1})\xrightarrow{d} \mathrm D^{-\infty}_c(M;\Omega_{\CO}^{0})\rightarrow0\rightarrow \cdots,
\]
 where we have used the identification (see \cite[Proposition 5.14]{CSW2} and 
 Remark \ref{rem:DFtop}) \[\mathrm D^{-\infty}_c(M;\Omega^r_{\CO})=(\Omega^r_\CO(M))' \quad \quad \text{($r\in\BN$)}\] of LCS. Recall that the  map $d$
 satisfies the condition
\be \label{eq:gradeddualofd}
\la d(\eta), \omega\ra=(-1)^{r+1}\la \eta, d(\omega)\ra
\ee 
for all $\eta\in  \mathrm D^{-\infty}_c(M;\Omega_{\CO}^{r+1})$ and $\omega\in \Omega^r_\CO(M)$.
By taking the quasi-completed inductive tensor product with $\iota^\bullet(E)$, we have  a
topological cochain complex
 \be \label{eq:D-inftyc}
\mathrm D^{-\infty}_c(M;\Omega_{\CO}^{-{\bullet}})
\widetilde\otimes_{\mathrm{i}}\,\iota^\bullet(E)\ee given by 
\[
 \cdots\rightarrow  \mathrm D^{-\infty}_c(M;\Omega_{\CO}^2)\wt\otimes_{\mathrm{i}}\, E\xrightarrow{d}
	\mathrm D^{-\infty}_c(M;\Omega_{\CO}^1)\wt\otimes_{\mathrm{i}}\, E\xrightarrow{d} \mathrm D^{-\infty}_c(M;\Omega_{\CO}^0)\wt\otimes_{\mathrm{i}}\, E\rightarrow0\rightarrow  \cdots.\]

Note that  the assignment \[\mathrm D^{-\infty}_c(\Omega_{\CO}^r)_E:\  U\mapsto  \mathrm  D^{-\infty}_c(U;\Omega_{\CO}^r)\wt\otimes_{\mathrm{i}}\, E
 \quad(\text{$U$ is an open subset of $M$})\] forms a precosheaf  of complex vector spaces  on $M$. 
 \begin{prpd}
    Suppose $E$ is a barreled DF space. Then for every $r\in \BN$, \be \label{eq:tOmegarOE}
\mathrm D^{-\infty}_c(M;\Omega_{\CO}^r)\wt\otimes_{\mathrm{i}}\, E=
\mathrm D^{-\infty}_c(M;\Omega_{\CO}^r)\wh\otimes_{\mathrm{i}}\, E
=\mathrm D^{-\infty}_c(M; \Omega^r_{\CO}; E)
\ee
as LCS. Furthermore, we have that  \be \label{eq:DcO-r=r}
 \mathrm D^{-\infty}_c(\Omega_{\CO}^r)_E=\mathrm D^{-\infty}_c(\Omega^r_{\CO};E)\ee is a cosheaf of complex vector spaces.
 \end{prpd}
 \begin{proof}
As the sheaf $\Omega^r_\CO$ of $\CO$-modules  is locally free of finite rank, the first assertion
follows from  \cite[Proposition 5.20]{CSW2}. The second assertion is a direct consequence of the first one.
 \end{proof}
 For $r\in \BN$ and two open subsets $V\subset U$ of $M$, it is easy to check that
\be \label{eq:dcommutesres} d\circ \mathrm{ext}_{U,V}=
\mathrm{ext}_{U,V}\circ d:\ \mathrm D^{-\infty}_c(V;\Omega_{\CO}^r)\wt\otimes_{\mathrm{i}}\, E\rightarrow \mathrm D^{-\infty}_c(U;\Omega_{\CO}^r)\wt\otimes_{\mathrm{i}}\, E.
\ee
Then we have a cochain complex
\[ \mathrm D^{-\infty}_c(\Omega_{\CO}^{-\bullet})_ E:\ \cdots \xrightarrow{d}\mathrm D^{-\infty}_c(\Omega_{\CO}^1)_ E\xrightarrow{d}\mathrm D^{-\infty}_c(\Omega_{\CO}^0)_E\rightarrow0\rightarrow 0\rightarrow \cdots\]
of precosheaves on $M$. We also write
 \[\mathrm D_c^{-\infty}(\Omega_{\CO}^{-\bullet}):=\mathrm D_c^{-\infty}(\Omega_{\CO}^{-\bullet})_ \C.\]
\begin{dfn}
   We call the complex $\mathrm D^{-\infty}_c(\Omega_{\CO}^{-\bullet})_ E$ the de Rham complex for $(M,\CO)$ with coefficients in  compactly
supported $E$-valued formal distributions.
\end{dfn}


Recall that $\mathrm D^{\infty}_c(\Omega_{\CO}^r)$
is naturally a subcosheaf of 
$\mathrm D^{-\infty}_c(\Omega_{\CO}^r)$ 
(see \cite[Corollary 5.16]{CSW2}).
We have the following result, which will be proved in this subsection.

\begin{prpd}\label{prop:defdoncsdc}
For every  $r\in \BN$,  by taking the restriction, the coboundary map \eqref{eq:gradeddualofd} on $\mathrm D^{-\infty}_c(\Omega_{\CO}^{-\bullet})$ induces a well-defined  continuous linear map
\be \label{eq:defdoncsdc}d=d|_{\mathrm D^{\infty}_c(M;\Omega_{\CO}^{r+1})}:\  \mathrm D^{\infty}_c(M;\Omega_{\CO}^{r+1})\rightarrow \mathrm D^{\infty}_c(M;\Omega_{\CO}^r).\ee
\end{prpd}

Proposition \ref{prop:defdoncsdc} shows that by taking the restriction, $\mathrm D^{-\infty}_c(\Omega_{\CO}^{-\bullet})$ induces a  cochain complex
\be \label{eq:comsupderhamO}
\mathrm D^{\infty}_c(\Omega_{\CO}^{-\bullet}):
\, \cdots\rightarrow  \mathrm D^{\infty}_c(\Omega_{\CO}^2)\xrightarrow{d} \mathrm D^{\infty}_c(\Omega_{\CO}^1)\xrightarrow{d} \mathrm D^{\infty}_c(\Omega_{\CO}^0)\rightarrow0\rightarrow 0\rightarrow \cdots
\ee
of cosheaves on $M$.
By  taking the global sections in \eqref{eq:comsupderhamO}, one obtains a topological cochain complex
\be\label{eq:comsupderham}
	\cdots\rightarrow  \mathrm D^{\infty}_c(M;\Omega_{\CO}^2)\xrightarrow{d} \mathrm D^{\infty}_c(M;\Omega_{\CO}^1)\xrightarrow{d} \mathrm D^{\infty}_c(M;\Omega_{\CO}^0)\rightarrow 0\rightarrow \cdots
\ee
of complete reflexive LCS, to be denoted by $\mathrm D^{\infty}_c(M;\Omega_{\CO}^{-\bullet})$. 

Moreover, by taking the quasi-completed inductive tensor product with $\iota^\bullet(E)$, we have a 
topological cochain complex 
\be\label{eq:Dinftyc}
\mathrm D^{\infty}_c(M;\Omega_{\CO}^{-\bullet})\widetilde\otimes_{\mathrm{i}}\,\iota^\bullet(E)\ee
of quasi-complete  LCS given by 
\[\cdots\rightarrow  \RD^{\infty}_c(M;\Omega_{\CO}^2)\widetilde\otimes_{\mathrm{i}}\,E\xrightarrow{d} \RD^{\infty}_c(M;\Omega_{\CO}^1)\widetilde\otimes_{\mathrm{i}}\,E\xrightarrow{d} \RD^{\infty}_c(M;\Omega_{\CO}^0)\widetilde\otimes_{\mathrm{i}}\,E\rightarrow 0\rightarrow \cdots. \]
 The assignment \[\mathrm D^{\infty}_c(\Omega_{\CO}^r)_E:\  U\mapsto  \mathrm  D^{\infty}_c(U;\Omega_{\CO}^r)\wt\otimes_{\mathrm{i}}\, E
 \quad(\text{$U$ is an open subset of $M$})\] forms a precosheaf  of complex vector spaces  on $M$, and we have a cochain complex
\[ \mathrm D^{\infty}_c(\Omega_{\CO}^{-\bullet})_ E:\ \cdots \xrightarrow{d}\mathrm D^{\infty}_c(\Omega_{\CO}^1)_ E\xrightarrow{d}\mathrm D^{\infty}_c(\Omega_{\CO}^0)_E\rightarrow0\rightarrow 0\rightarrow \cdots\]
of precosheaves on $M$, which is a subcomplex of 
$\mathrm D^{\infty}_c(\Omega_{\CO}^{-\bullet})_E$.
\begin{dfn}
We call $\mathrm D^{\infty}_c(\Omega_{\CO}^{-{\bullet}})_E$ the  de Rham complex for $(M,\CO)$ with coefficients in  compactly supported $E$-valued formal densities.
\end{dfn}
Now we are going to prove Proposition \ref{prop:defdoncsdc}. For this purpose, we  begin by describing the LCS $\mathrm D^{\infty}_c(M;\Omega_{\CO}^r)$ when $M=N^{(k)}$, where $n,k\in \BN$ and, without loss of generality, $N$ is a nonempty open submanifold of $\R^n$. Let $r=0,1,\dots,n+k$.  Recall from Example \ref{ex:DF} that, in this case, 
\[\Omega^r_\CO(M)=\bigoplus_{(I,J)\in \Lambda_{n,k}^r} \mathrm{C}^\infty(N)[[y_1,y_2,\dots,y_k]] dx_I dy_J.
\]
For every \[(I,J)=((i_1,i_2,\dots,i_s),(j_1,j_2,\dots,j_{n+k-r-s}))\in \Lambda_{n,k}^{n+k-r},\]
and \[\tau=\sum_{L_2=(l_1,l_2,\dots,l_k)\in \BN^k} \tau_{L_2}\cdot (y^*)^{L_2}\in \mathrm{D}_c^\infty(N)[y_1^*,y_2^*,\dots,y_k^*]\] with $\tau_{L_2}\in \mathrm{D}^\infty_c(U)$ and $(y^*)^{L_2}=(y_1^*)^{l_1}(y_2^*)^{l_2}\cdots(y_k^*)^{l_k}$, 
we define
\[\tau dx_{I}^*d_{J}^*=\tau dx^*_{i_1}\wedge dx_{i_2}^*\wedge \cdots\wedge dx^*_{i_s}\wedge dy^*_{j_1}\wedge dy^*_{j_2}\wedge
\cdots \wedge dy^*_{j_{n+k-r-s}}\]
to be the continuous linear functional on $\Omega^r_\CO(M)$ such that
\be \label{eq:Pdualpair}\la f dx_{I'} dy_{J'}, \tau dx_{I}^*dy_{J}^*\ra:= \la f, \tau \ra \cdot \varepsilon((I',J'),(I,J)),\ee
where $(I',J')\in \Lambda_{n,k}^r$,  \[f=\sum_{L_1=(l_1,l_2,\dots,l_k)\in \BN^k} f_{L_1} y^{L_1}\in \mathrm{C}^\infty(N)[[y_1,y_2,\dots,y_k]]\quad (y^{L_1}=y_1^{l_1}y_2^{l_2}\cdots y_k^{l_k}), \] 
$\varepsilon((I',J'),(I,J))$ is the constant determined by the equality
\begin{eqnarray*}
&&dx_{I'}dy_{J'}\wedge dx_{I}dy_{J}\\
&=&\varepsilon((I',J'),(I,J)) \cdot dx_1\wedge dx_2\wedge \cdots \wedge dx_n\wedge dy_1\wedge dy_2\wedge \cdots \wedge dy_k,
\end{eqnarray*}
and  \be \label{eq:pairing2} \la f, \tau \ra
=\sum_{L_1\in \BN^k} L_1! \cdot \int_N f_{L_1}\tau_{L_1}.\ee  

Let $\CF$ be a sheaf of $\CO$-modules. Recall that $\mathcal{D}\mathrm{iff}_c(\CF,\underline{\mathcal{D}})$ denotes the cosheaf 
\[U\mapsto \mathrm{Diff}_c(\CF|_U,\underline{\mathcal{D}}) \quad \text{($U$ is an open subset of $M$)}\]
of compactly supported differential operators (see \cite[Section 3.2]{CSW1}) from $\CF$ to the sheaf $\underline{\mathcal{D}}$ of (complex-valued) smooth densities on $N$, and recall from \cite[Section 2.1]{CSW2} that there is a homomorphism  \[\rho=\{\rho_U\}_{\text{$U$ is an open subset of $M$}}:\quad\mathcal{D}\mathrm{iff}_c(\CF,\underline{\mathcal{D}})
\rightarrow \mathrm{D}^\infty_c(\CF)\] of cosheaves given by 
\begin{equation}\begin{split}\label{eq:defcomsuppden}
\rho_U:\ \mathrm{Diff}_{c}(\CF|_U,\underline{\mathcal{D}}|_U)&\rightarrow \mathrm{D}^{\infty}_{c}(U;\CF),\quad \text{($U$ is an open subset of $M$)} \\
D&\mapsto \left(u\mapsto  \int_{U} D(u)\right).
\end{split}\end{equation}

Recall  from \cite[Section 2.2]{CSW2} that $(\mathcal{C}(M),\preceq)$ is a  directed set, where
\be\label{eq:C(M)} \mathcal{C}(M):=\{(K,s)\mid \text{$K$ is a compact subset of $M$ and $s\in \BN$}\},\ee
and for two pairs $(K,s),(K',s')\in \mathcal{C}(M)$, 
\[\text{$(K,s)\preceq(K',s')$ if and only if $K\subset K'$ and $s\le s'$.}\] 
For  $(K,s)\in \mathcal{C}(M)$, recall from \cite[Section 2.2]{CSW2} that \[\mathrm{D}^\infty_{K,s}(M;\CF):=\rho_M(\mathrm{Diff}_{K,s}(\CF,\underline{\mathcal{D}}))\subset \mathrm{D}^\infty_{c}(M;\CF)\]
is endowed with the quotient topology of $\mathrm{Diff}_{K,s}(\CF,\underline{\mathcal{D}})$, and 
\[
\mathrm{D}^\infty_{c}(M;\CF)=\varinjlim_{(K,s)\in \mathcal{C}(M)}\mathrm{D}^\infty_{K,s}(M;\CF)\]
is endowed with the inductive limit topology. 
Here $\mathrm{Diff}_{K,s}(\CF,\underline{\mathcal{D}})$ denotes the subspace of $\mathrm{Diff}_c(\CF,\underline{\mathcal{D}})$  consisting of all the
differential operators  with order $\le s$ and supported in $K$ (see  \cite[Sections 2.1 and 2.2]{CSW2} for details).



\begin{lemd}\label{lem:charOmegaonkr} 
Assume that $M=N^{(k)}$. Then for every $(K,s)\in \mathcal{C}(M)$,
the map
\be\begin{split} \label{eq:topiosondnck-ru2}
\bigoplus_{(I,J)\in \Lambda_{n,k}^{n+k-r}} \mathrm{D}^\infty_K(N)[y_1^*,y_2^*,\dots,y_k^*]_{\leq s}
\rightarrow \mathrm D^{\infty}_{K,s}(M;\Omega_{\CO}^r),\\
(\tau_{I,J})\mapsto \sum \tau_{I,J}dx^*_{I} dy^*_{J}
\end{split}
\ee is a well-defined topological linear isomorphism. Moreover, the map
\be\label{eq:topiosondnck-ru}
\bigoplus_{(I,J)\in \Lambda_{n,k}^{n+k-r}} \mathrm{D}^\infty_c(N)[y_1^*,y_2^*,\dots,y_k^*]
\rightarrow \mathrm D^{\infty}_{c}(M;\Omega_{\CO}^r)
\ee
induced by \eqref{eq:topiosondnck-ru2} is also a topological linear isomorphism.
\end{lemd}
\begin{proof} 
 By using \eqref{eq:DFid}, we have that 
\begin{eqnarray*}
&&\mathrm D^{\infty}_{K,s}(M;\Omega_{\CO}^r)=\rho_M(\mathrm{Diff}_{K,s}(\Omega^r_{\CO},\underline{\mathcal{D}}))\\
&\cong&\bigoplus_{\Lambda_{n,k}^{n+k-r}}\rho_M( \mathrm{Diff}_{K,s}(\CO,\underline{\mathcal{D}}))
=\bigoplus_{\Lambda_{n,k}^{n+k-r}}\mathrm D^{\infty}_{K,s}(M;\CO)\end{eqnarray*} as LCS  for every $(K,s)\in \mathcal{C}(M)$. 
This, together with \cite[Proposition 2.19]{CSW2}, implies that \eqref{eq:topiosondnck-ru2} 
is a well-defined topological linear isomorphism.
By \cite[(2.36)]{CSW2}, the map  \eqref{eq:topiosondnck-ru} is also a topological linear isomorphism.
\end{proof}


For every $i=1,2,\dots,n$, define the continuous linear map
\[\partial_{x_i^*}:\mathrm{D}^\infty_c(N)\rightarrow \mathrm{D}^\infty_c(N),\quad
f\mathrm{d}x\mapsto \partial_{x_i}(f)\mathrm{d}x,\]
where $f\in \RC_c^\infty(N)$ and  $\mathrm{d}x$ is the Lebesgue measure.
And, for every   $j=1,2,\dots,k$, define the continuous linear map
\[m_{y_j^*}:\C[y_1^*,y_2^*,\dots,y_k^*]\rightarrow \C[y_1^*,y_2^*,\dots,y_k^*],\quad
 (y^*)^L\mapsto y_j^*\cdot (y^*)^L\quad (L\in \BN^k).\]
Extend  the linear endomorphisms $\partial_{x_i^*}$ and $m_{y_j^*}$ to
 $\mathrm{D}^\infty_c(N)[y_1^*,y_2^*,\dots,y_k^*]$ in an obvious way.

 \begin{lemd}\label{lem:comdoncsdc} 
 Proposition \ref{prop:defdoncsdc} holds when $M=N^{(k)}$ with $N$ a nonempty open submanifold of $\R^n$ and $n,k\in \BN$.
 \end{lemd}
 \begin{proof} 
 Let $\tau\in \mathrm{D}^\infty_c(N)[y_1^*,y_2^*,\dots,y_k^*]$, $r=0,1,\dots,n+k-1$, and
 \[(I,J)=((i_1,i_2,\dots,i_s),(j_1,j_2,\dots,j_{n+k-r-1-s}))\in \Lambda_{n,k}^{n+k-r-1}.\]
It is easy to check  that
\be \label{eq:comexplicitd}\la f,\partial_{x_i^*}(\tau)\ra=\la -\partial_{x_i}(f),\tau\ra\quad \text{and}\quad
\la f,m_{y_j^*}(\tau)\ra=\la \partial_{y_j}(f),\tau\ra,\ee
for every $f\in \RC^\infty(N)[[y_1,y_2,\dots,y_k]]$.
 
Using \eqref{eq:gradeddualofd} and \eqref{eq:comexplicitd}, we have  the following equality by a straightforward computation:
 \begin{eqnarray}\label{eq:comdoncsdc}&&
d(\tau  dx_I^* dy_J^*) \\&=&\sum_{i\in \{1,2,\dots,n\}\backslash \{ i_1,i_2,\dots, i_s\} } \partial_{x_i^*}(\tau) dx_a\wedge dx_I^*dy_J^*\notag\\
&+&\sum_{j\in  \{1,2,\dots,k \} \backslash \{ j_1,j_2,\dots,j_{n+k-r-1-s}\} }(-1)^{s-1}m_{y_j^*}(\tau) dx_I^*\wedge dy_{j}^*\wedge dy_J^*. \notag
\end{eqnarray}
This, together with Lemma \ref{lem:charOmegaonkr}, implies the lemma.
\end{proof}

\vspace{3mm}

\noindent\textbf{Proof of Proposition \ref{prop:defdoncsdc}:} Let $\eta\in \mathrm D^{\infty}_c(M;\Omega_{\CO}^{r+1})$. Take 
$D\in \mathrm{Diff}_c(\Omega^{r+1}_{\CO},\underline{\mathcal{D}})$ such that $\eta=\rho_M(D)$, and  pick  a finite family $\{U_1,U_2,\dots,U_t\}$ of charts  of $M$ that covers  $K:=\mathrm{supp}\,D$ in $M$. Let  $\{f_0,f_1,\dots,f_t\}$ be a partition of unity of $M$ subordinate to the cover $\{U_0:=M\setminus K,U_1,\dots,U_t\}$ (see \cite[Proposition 2.3]{CSW1}).
Then we have that
\begin{eqnarray*}
&&d(\eta)\\&=&d\left(\sum_{i=1}^t \eta\circ f_i \right)\\ &=&\sum_{i=1}^t d (\mathrm{ext}_{M,U_i} ((\eta\circ f_i)|_{U_i}))\quad (\text{by \cite[Lemma 2.12]{CSW2}})\\
&=&\sum_{i=1}^t\mathrm{ext}_{M,U_i} \circ d ((\eta\circ f_i)|_{U_i})\in\mathrm D_{c}^{\infty}(M;\Omega_{\CO}^r)\quad (\text{by \eqref{eq:dcommutesres} and Lemma \ref{lem:comdoncsdc}}).
\end{eqnarray*}
This implies that the linear  map \eqref{eq:defdoncsdc} is well-defined.

For the continuity of the map \eqref{eq:defdoncsdc}, let us take an atlas $\{V_\alpha\}_{\alpha\in \Gamma}$ of $M$.
From \eqref{eq:dcommutesres}, we have the  commutative diagram
\[
\begin{CD}
	\mathrm D^{\infty}_c(M;\Omega_{\CO}^{r+1})@> d >> \mathrm D^{\infty}_c(M;\Omega_{\CO}^r)\\
	@A  \oplus_{\alpha\in \Gamma}\mathrm{ext}_{M,V_\alpha}AA       @AA\oplus_{\alpha\in \Gamma} \mathrm{ext}_{M,V_\alpha} A   \\
	\bigoplus_{\alpha\in \Gamma}\mathrm D^{\infty}_c(V_{\alpha};\Omega_{\CO}^{r+1})@> \oplus_{\alpha\in \Gamma} d >> \bigoplus_{\alpha\in \Gamma}
\mathrm D^{\infty}_c(V_{\alpha};\Omega_{\CO}^r). \\
\end{CD}
\]
 \cite[Proposition 2.15]{CSW2}  implies that the left and right vertical arrows are both continuous, open and surjective linear maps. 
Meanwhile, Lemma \ref{lem:comdoncsdc} implies that the bottom horizontal arrow is a continuous linear map. Thus,
the top horizontal arrow is a continuous linear map as well.
This completes the proof of Proposition \ref{prop:defdoncsdc}.
\qed

\vspace{3mm}

\begin{remarkd}\label{rem:usuPD}
Assume that $M=N^{(0)}$ with $N$ a nonempty open submanifold of $\R^n$ for some $n\in \BN$.
For $r=0,1,\dots,n$, the space of
compactly supported differential $r$-forms on $N$ is usually defined as
\[\Omega_{\CO,c}^r(N):=\bigoplus_{I\in \Lambda_n^r} \RC^\infty_c(N) dx_I\subset
\Omega^r_\CO(N).\]
With the obvious identification
\[ 
\Omega_{\CO,c}^{n-r}(N)\rightarrow\mathrm D^{\infty}_c(N;\Omega_{\CO}^r),\quad \quad  f dx_I\mapsto (f\mathrm{d}x) dx^*_I,
\]
one concludes from \eqref{eq:comdoncsdc} that the coboundary map \[d:\mathrm D^{\infty}_c(N;\Omega_{\CO}^r)\rightarrow \mathrm D^{\infty}_c(N;\Omega_{\CO}^{r-1})\]
 is nothing but the usual coboundary map
\[d:\Omega_{\CO,c}^{n-r}(N)\rightarrow \Omega_{\CO,c}^{n-r+1}(N)\] (defined by taking the restriction of $d:\Omega^{n-r}_\CO(N)\rightarrow \Omega^{n-r+1}_\CO(N)$).
Moreover, the pairing \eqref{eq:Pdualpair} coincides with the usual Poincar\'e dual pairing
\[\begin{array}{rcl}
\Omega^r_\CO(N)\times \Omega_{\CO, c}^{n-r}(N)&\rightarrow& \C, \\ (f dx_I, g dx_J)
&\mapsto& \int_N fg\, dx_I\wedge dx_J.\end{array}\]
\end{remarkd}

Finally, by considering the transpose of \eqref{eq:comsupderham}, one obtains a  topological cochain complex
\bee\label{eq:trancomsupderham}
\mathrm C^{-\infty}(M;\Omega_{\CO}^{\bullet}):
\,\cdots\rightarrow 0 \rightarrow \mathrm C^{-\infty}(M;\Omega_{\CO}^{0})\xrightarrow{d} \mathrm C^{-\infty}(M;\Omega_{\CO}^{1})\xrightarrow{d} \mathrm C^{-\infty}(M;\Omega_{\CO}^{2})\xrightarrow{d} \cdots
\eee
of complete nuclear reflexive  LCS, where (see \cite[(3.1)]{CSW2})
\[\mathrm C^{-\infty}(M;\Omega^r_{\CO})=(\mathrm D^{\infty}_c(M;\Omega_{\CO}^r))'
\] for $r\in \BN$.
Furthermore, by taking the quasi-completed projective  tensor product with $\iota^\bullet(E)$, we have  a  topological cochain complex
\be\label{eq:C-infty}\mathrm C^{-\infty}(M;\Omega_{\CO}^{\bullet})\widetilde\otimes\, \iota^\bullet(E)\ee
of quasi-complete   LCS as follows:
\[
	\cdots\rightarrow 0\rightarrow 0 \rightarrow \mathrm C^{-\infty}(M;\Omega_{\CO}^0)\widetilde\otimes E\xrightarrow{d} \mathrm C^{-\infty}(M;\Omega_{\CO}^{1})\widetilde\otimes E\xrightarrow{d} \mathrm C^{-\infty}(M;\Omega_{\CO}^2)\widetilde\otimes E\xrightarrow{d} \cdots.
\]
 Since $\Omega^r_\CO$ is locally free of finite rank, it follows from  \cite[Proposition 3.5]{CSW2} that
 \[\mathrm C^{-\infty}(M;\Omega^r_{\CO})\widetilde\otimes E=\mathrm C^{-\infty}(M;\Omega^r_{\CO};E).
 \]
 This, together with \cite[Lemma 3.3]{CSW2}, implies that the assignment
 \[\mathrm C^{-\infty}(\Omega_{\CO}^r)_E:\ U\mapsto \mathrm C^{-\infty}(U;\Omega_{\CO}^{r})\widetilde\otimes E
 \]
 is a sheaf of $\CO$-modules over $M$.

 For every $r\in \BN$, it is easy to check  that \[d\circ \mathrm{res}_{V,U}=\mathrm{res}_{V,U}\circ d:\  \mathrm C^{-\infty}(U;\Omega_{\CO}^r)\widetilde\otimes E\rightarrow\mathrm C^{-\infty}(V;\Omega_{\CO}^{r+1})\widetilde\otimes E,\] where  $V\subset U$ are  two open subsets of $M$.
Consequently, there is a cochain complex
\begin{eqnarray*}
	\cdots\rightarrow 0\rightarrow 0\rightarrow \mathrm C^{-\infty}(\Omega_{\CO}^0)_ E\xrightarrow{d} \mathrm C^{-\infty}(\Omega_{\CO}^1)_ E\xrightarrow{d} \mathrm C^{-\infty}(\Omega_{\CO}^2)_ E\xrightarrow{d} \cdots
\end{eqnarray*}
of sheaves on $M$, to be denoted by $\mathrm C^{-\infty}(\Omega_{\CO}^{\bullet})_E$.
\begin{dfn}
    We call $\mathrm C^{-\infty}(\Omega_{\CO}^{\bullet})_E$  the de Rham complex for $(M,\CO)$ with coefficients in $E$-valued formal generalized functions.
\end{dfn}

\section{Tensor products of de Rham complexes}\label{sec:tenderham}
Throughout this section, let $(M_1,\CO_1)$, $(M_2,\CO_2)$ be two formal manifolds, and let $(M_3,\CO_3)$ be  the product of them (see \cite[Theorem 6.18]{CSW1}).
We consider  the tensor products of those de Rham complexes
introduced in the last section.

\subsection{Tensor products of de Rham complexes I}
The main goal of this subsection is to prove the following theorem.  
We refer to Appendix \ref{appendixC2} for the notions of tensor products of topological cochain complexes, and Section \ref{appendixC1} for the notions of continuous complex maps, topological complex isomorphisms, topological homotopy equivalences, and etc.

\begin{thmd}\label{thm:tensorofcomplex} Let $(M_1,\CO_1)$ and $(M_2,\CO_2)$ be two formal manifolds and let
$(M_3,\CO_3)$
be the product of them. Then we have the identifications
\begin{eqnarray*} \Omega^{\bullet}_{\CO_3}(M_3)&=&\Omega^{\bullet}_{\CO_1}(M_1)\widetilde\otimes\,\Omega^{\bullet}_{\CO_2}(M_2)\\
&=&\Omega^{\bullet}_{\CO_1}(M_1)\widehat\otimes\, \Omega^{\bullet}_{\CO_2}(M_2)\end{eqnarray*}
and
\begin{eqnarray*} \mathrm D^{-\infty}_c(M_3;\Omega_{\CO_3}^{-\bullet})&=& \mathrm D^{-\infty}_c(M_1;\Omega_{\CO_1}^{-\bullet})\widetilde\otimes_{\mathrm{i}}\, \mathrm D^{-\infty}_c(M_2;\Omega_{\CO_2}^{-\bullet})\\
&=&\mathrm D^{-\infty}_c(M_1;\Omega_{\CO_1}^{-\bullet})\widehat\otimes_{\mathrm{i}}\,\mathrm D^{-\infty}_c(M_2;\Omega_{\CO_2}^{-\bullet})
\end{eqnarray*}
of topological cochain complexes. 
\end{thmd}

Let 
\[ p_i:\,(M_3,\CO_3)\rightarrow(M_i,\CO_i)\quad (i=1,2)\] be the projection map.
By 
Proposition \ref{pro}, there is  a continuous  complex map
\[p_i^{\natural}:\Omega^{\bullet}_{\CO_i}(M_i)\rightarrow \Omega^{\bullet}_{\CO_3}(M_3).\]
On the other hand, for $r_1,r_2\in\BN$, we have a continuous linear map 
\[\wedge:\  \Omega^{r_1}_{\CO_3}(M_3)\widetilde \otimes \Omega^{r_2}_{\CO_3}(M_3)\rightarrow
 \Omega^{r_1+r_2}_{\CO_3}(M_3)  \]
  induced by the wedge product \eqref{eq:wedge}.
  
Consider the composition  map
\be\label{eq:defPsi}\Psi:\ \Omega^{r_1}_{\CO_1}(M_1)\widetilde \otimes \Omega^{r_2}_{\CO_2}(M_2)\xrightarrow{p_1^\natural\otimes p_2^\natural}
\Omega^{r_1}_{\CO_3}(M_3)\widetilde \otimes \Omega^{r_2}_{\CO_3}(M_3)\xrightarrow{\ \wedge\ }
 \Omega^{r_1+r_2}_{\CO_3}(M_3).\ee 
Note that 
 \be\label{eq:p_1p_2}\Psi(\omega_{1}\otimes \omega_{2})=p_1^{\natural}(\omega_{1})\wedge p_2^{\natural}(\omega_{2})\quad \quad (\omega_{1}\in \Omega^{r_1}_{\CO_1}(M_1),\,
   \omega_{2} \in \Omega_{\CO_2}^{r_2}(M_2)).\ee
By Proposition \ref{prop:EvaluedDF} and  Remark \ref{rem:DFtop}, we have that
\[\Omega^{r_1}_{\CO_1}(M_1)\widetilde\otimes \Omega^{r_2}_{\CO_2}(M_2)=\Omega^{r_1}_{\CO_1}(M_1;\Omega^{r_2}_{\CO_2}(M_2))=\Omega^{r_1}_{\CO_1}(M_1)\widehat\otimes \Omega^{r_2}_{\CO_2}(M_2).\]
Thus \eqref{eq:defPsi}  yields a continuous map  (see \eqref{eq:conmap})
   \be \label{eq:Psi}
  \Psi:\,\Omega^{\bullet}_{\CO_1}(M_1)\widetilde\otimes \Omega^{\bullet}_{\CO_2}(M_2)=\Omega^{\bullet}_{\CO_1}(M_1)\widehat\otimes \Omega^{\bullet}_{\CO_2}(M_2)\longrightarrow \Omega^{\bullet}_{\CO_3}(M_3)\ee
of degree $0$.

\begin{lemd} The map \eqref{eq:Psi} is a continuous complex map.
\end{lemd}
\begin{proof}
For $\omega_{1}\in \Omega^{r_1}_{\CO_1}(M_1)$ and
   $\omega_2 \in \Omega^{r_2}_{\CO_2}(M_2)$, using \eqref{eq:p_1p_2}, \eqref{eq:dwedge} and Proposition \ref{pro},  we have that
 \begin{eqnarray*}
d\circ \Psi(\omega_1\otimes \omega_2)
&=&d (p_1^{\natural}(\omega_{1}))\wedge p_2^{\natural}(\omega_{2})+(-1)^{r_1} p_1^{\natural}(\omega_1)\wedge d(p_2^{\natural}(\omega_{2}))\\
	&=& p_1^{\natural}(d(\omega_{1}))\wedge p_2^{\natural}(\omega_{2})+(-1)^{r_1} p_1^{\natural}(\omega_{1})\wedge p_2^{\natural}(d(\omega_{2}))\\
	&=&\Psi(d(\omega_{1})\otimes\omega_{2}+(-1)^{r_1}\omega_{1}\otimes d(\omega_{2})) \\
	&=&\Psi\circ d(\omega_{1}\otimes\omega_{2}).
\end{eqnarray*}
This finishes the proof.
\end{proof}

Let $n,k,r,s\in \BN$. 
For 
\[
I_1=(i_1,i_2,\dots,i_{r})\in \Lambda_{n}^{r}\quad 
\text{and}\quad I_2=(i_1',i_2',\dots,i_{s}')\in \Lambda_{k}^{s},
\]
put
\be \label{eq:defn+I}
(I_1,n+I_2):=
(i_1,i_2,\dots,i_{r},n+i_1',n+i_2',\dots,n+i_{s}')\in \Lambda_{n+k}^{r+s}.
\ee

\begin{lemd}\label{lem:Psiisopre} 
Assume that 
$M_i=N_i^{(k_i)}$  for $i=1,2$, where $N_i$ is an open submanifold  of $\R^{n_i}$, and $n_i,k_i\in \BN$. Then   the continuous complex  map \eqref{eq:Psi}  is a topological complex isomorphism.
\end{lemd}
\begin{proof}Set \[N_3:=N_1\times N_2\subset \R^{n_3},\ n_3:=n_1+n_2\quad  \text{and}\quad  k_3:=k_1+k_2.\]
By \cite[Lemma 6.19]{CSW1}, the product formal manifold
\[(M_3,\CO_3)=(N_3,\CO_{N_3}^{(k_3)}).\]
Let $r\in \BN$. 
The lemma then follows from  the commutative diagram
\[\begin{CD}
	\bigoplus\limits_{r_1,r_2\in \BN; r_1+r_2=r}\left(\Omega^{r_1}_{\CO_1}(M_1)\widetilde \otimes \Omega^{r_2}_{\CO_2}(M_2)\right)	@>{\Psi} >>  \Omega^r_{\CO_3}(M_3) \\
		@V{\eqref{eq:DFid}}  VV           @V V{\eqref{eq:DFid}} V\\
	\prod\limits_{((I_1,J_1),(I_2,J_2))} \CO_1(M_1)\widetilde\otimes \CO_2(M_2) @ > {\widetilde\pi_r}> > \prod\limits_{(I,J)}
\CO_3(M_3),\\
	\end{CD}
	\]
	where  $((I_1,J_1),(I_2,J_2))$ runs over   \[\bigsqcup_{r_1,r_2\in \BN; r_1+r_2=r}\Lambda_{n_1,k_1}^{r_1}\times \Lambda_{n_2,k_2}^{r_2},\] $(I,J)$ runs over $\Lambda_{n_3,k_3}^r$,  and $\widetilde\pi_r$ is the topological linear isomorphism defined by
\bee
\{f_{((I_1,J_1),(I_2,J_2))} \} \mapsto \{\tilde{f}_{(I,J)}\},\quad (\tilde{f}_{(I,J)}:=(-1)^{s_1t_2} f_{\pi^{-1}_r((I,J))}). \eee
Here 
$\pi_r$ denotes the  bijection 
\begin{eqnarray*}
\pi_r:\ 
\bigsqcup_{r_1,r_2\in \BN; r_1+r_2=r}\Lambda_{n_1,k_1}^{r_1}\times \Lambda_{n_2,k_2}^{r_2}&\rightarrow& \Lambda_{n_3,k_3}^{r},\\
 ((I_1,J_1),(I_2,J_2))&\mapsto& ((I_1,n_1+I_2),(J_1,n_2+J_2)),
\end{eqnarray*} 
 and $s_1,t_2$ are the non-negative integers such that 
\[
\pi^{-1}_r((I,J))\in \left(\Lambda_{n_1}^{t_1}\times \Lambda_{k_1}^{s_1}\right)\times \left(\Lambda_{n_2}^{t_2}\times \Lambda_{k_2}^{s_2}\right)
\] for some  non-negative integers $s_2$ and $t_1$.
\end{proof}

In general, we have the following result. 
\begin{prpd}\label{propro} The continuous complex  map \eqref{eq:Psi}  is a topological complex isomorphism.
\end{prpd}
\begin{proof}

Let $V_1$ be a chart  of $M_1$ and $r\in \BN$.
 Proposition  \ref{prop:EvaluedDF} implies that
\[\Omega^r_{V_1}:\ U_2\mapsto \left(\Omega^r_{V_1}(U_2):=
\bigoplus_{r_1,r_2\in \BN; r_1+r_2=r}\left(\Omega^{r_1}_{\CO_1}(V_1)\widetilde\otimes \Omega^{r_2}_{\CO_2}(U_2)\right)\right)\]
is a sheaf of complex vector spaces over $M_2$, and that
\[\Omega^r_{M_2}:\ U_1\mapsto \left(\Omega^r_{M_2}(U_1):=
\bigoplus_{r_1,r_2\in \BN; r_1+r_2=r}\left(\Omega^{r_1}_{\CO_1}(U_1)\widetilde\otimes \Omega^{r_2}_{\CO_2}(M_2)\right)\right)\]
is a sheaf  of complex vector spaces over $M_1$.

For  open subsets $U_1\subset M_1$ and $U_2\subset M_2$, 
write 
\[\Psi_{U_1\times U_2}:\ \bigoplus_{r_1,r_2\in \BN; r_1+r_2=r}\left(\Omega^{r_1}_{\CO_1}(U_1)\widetilde\otimes \Omega^{r_2}_{\CO_2}(U_2)\right)
\rightarrow \Omega^r_{\CO_3}(U_1\times U_2)\]
for the canonical map.
These maps produce  a  sheaf homomorphism
\be\label{eq:OmegaV1} \{\Psi_{V_1\times U_2}\}_{U_2 \text{ is a open subset of } M_2}: \ 
\Omega^r_{V_1} \rightarrow (p_2|_{V_1\times M_2})_* \Omega^r_{\CO_3}(V_1\times M_2), 
\ee
as well as a sheaf homomorphism
\be \label{eq:OmegaV2} \{\Psi_{U_1\times M_2}\}_{U_1 \text{ is a open subset of } M_1}:\ 
\Omega^r_{M_2} \rightarrow (p_1)_* \Omega^r_{\CO_3}.
\ee

Lemma \ref{lem:Psiisopre}  implies that $\Psi_{V_1\times U_2}$ is an isomorphism for  every chart $U_2$ of $M_2$.
This forces that the sheaf homomorphism \eqref{eq:OmegaV1} is an isomorphism.
Thus $\Psi_{U_1\times M_2}$ is an isomorphism for every chart $U_1$ of $M_1$, and hence
 the sheaf homomorphism \eqref{eq:OmegaV2} is  an isomorphism as well. Then we see that the map \[\Psi:\bigoplus_{r_1,r_2\in \BN;r_1+r_2=r}\left(\Omega^{r_1}_{\CO_1}(M_1)\widetilde\otimes \Omega^{r_2}_{\CO_2}(M_2)\right) \rightarrow \Omega_{\CO_3}^r(M_3)\] is an isomorphism of complex vector spaces.

 Consider the commutative diagram 
\[\begin{CD}\bigoplus_{r_1,r_2}(\Omega^{r_1}_{\CO_1}(M_1)\wt\otimes\Omega^{r_2}_{\CO_2}(M_2))@>\Psi>> \Omega^r_{\CO_3}(M_3)\\@VVV @VVV\\\bigoplus_{r_1,r_2}\left(\left(\prod_{V_1}\Omega^{r_1}_{\CO_1}(V_1)\right)\wt\otimes\left(\prod_{V_2}\Omega^{r_2}_{\CO_2}(V_2)\right)\right)@>>>\prod_{V_1, V_2}\Omega_{\CO_3}^r(V_1\times V_2),\\\end{CD}\] where $r_1,r_2$ runs over all natural numbers such that $r_1+r_2=r$, $V_1$, $V_2$ respectively runs over all charts of $M_1$ and $M_2$,  and the bottom horizontal arrow is the composition of the topological linear isomorphism (see \cite[Lemma 6.6]{CSW1})\[\bigoplus_{r_1,r_2}\left(\left(\prod_{V_1}\Omega^{r_1}_{\CO_1}(V_1)\right)\wt\otimes\left(\prod_{V_2}\Omega^{r_2}_{\CO_2}(V_2)\right)\right) \rightarrow \prod_{V_1,V_2}\left(\bigoplus_{r_1,r_2}\left(\Omega^{r_1}_{\CO_1}(V_1)\wt\otimes\Omega^{r_2}_{\CO_2}(V_2)\right)\right)\]
and the topological linear isomorphism (see Lemma \ref{lem:Psiisopre})
\[\prod_{V_1, V_2}\left(\bigoplus_{r_1,r_2}\Omega^{r_1}_{\CO_1}(V_1)\wt\otimes\Omega^{r_2}_{\CO_2}(V_2)\right)\xrightarrow{\prod\Psi_{V_1\times V_2}}\prod_{V_1, V_2}\Omega_{\CO_3}^r(V_1\times V_2).\]

Lemma \ref{emd} and \cite[Proposition 43.7]{Tr} implies that the left vertical arrow is a linear topological embedding. By Lemma \ref{emd}, the right vertical arrow is also a linear topological embedding. Thus the isomorphism \[\Psi:\bigoplus_{r_1,r_2\in \BN;r_1+r_2=r}\left(\Omega^{r_1}_{\CO_1}(M_1)\widetilde\otimes \Omega^{r_2}_{\CO_2}(M_2)\right) \rightarrow \Omega_{\CO_3}^r(M_3)\] of complex vector spaces is a topological isomorphism. This proves the lemma.
\end{proof}

\noindent\textbf{Proof of Theorem \ref{thm:tensorofcomplex}:} The first assertion of of Theorem \ref{thm:tensorofcomplex}  follows from Proposition \ref{propro}. It remains to prove the second part. 

By \eqref{eq:tOmegarOE}, 
it follows from  \cite[Corollary 5.15, Proposition 5.14 and Lemma A.10]{CSW2} that
 \[\mathrm D^{-\infty}_c(M_1;\Omega_{\CO_1}^{-\bullet})\widetilde\otimes_{\mathrm{i}}\,\mathrm D^{-\infty}_c(M_2;\Omega_{\CO_2}^{-\bullet})=\mathrm D^{-\infty}_c(M_1;\Omega_{\CO_1}^{-\bullet})\widehat\otimes_{\mathrm{i}}\,\mathrm D^{-\infty}_c(M_2;\Omega_{\CO_2}^{-\bullet}).
\]
Furthermore, since $\Omega^{r_1}_{\CO_1}(M_1)$ and $\Omega^{r_2}_{\CO_2}(M_2)$ ($r_1,r_2\in \BN$) are both  products of nuclear Fr\'echet spaces by \eqref{eq:DFid} and \cite[Proposition 4.8]{CSW1},
 the canonical map (see \cite[Lemma A.12]{CSW2})
 \begin{eqnarray*}
\tau:\ (\Omega^{r_1}_{\CO_1}(M_1))'\widehat\otimes_{\mathrm{i}}\, (\Omega^{r_2}_{\CO_2}(M_2))'&\rightarrow&
\left(\Omega^{r_1}_{\CO_1}(M_1)\widehat\otimes\, \Omega^{r_2}_{\CO_2}(M_2)\right)',\\
 \eta_1\otimes \eta_2&\mapsto& (\omega_1\otimes \omega_2\mapsto (-1)^{r_1r_2}\eta_1(\omega_1)\eta_2(\omega_2))
\end{eqnarray*} is a topological linear isomorphism.
This isomorphism obviously extends to a topological complex isomorphism
\begin{eqnarray*}
\tau:\ \mathrm D^{-\infty}_c(M_1;\Omega_{\CO_1}^{-\bullet})\widehat\otimes_{\mathrm{i}}\, \mathrm D^{-\infty}_c(M_2;\Omega_{\CO_2}^{-\bullet})\rightarrow {}^t\left(\Omega^\bullet_{\CO_1}(M_1)\widehat\otimes\Omega^\bullet_{\CO_2}(M_2)\right),
\end{eqnarray*}
where the latter complex denotes the transpose of
$\Omega^\bullet_{\CO_1}(M_1)\widehat\otimes\Omega^\bullet_{\CO_2}(M_2)$ (see Appendix \ref{appendixC3}).

On the other hand, by transposing
 the inverse  \[\Psi^{-1}:\Omega^\bullet_{\CO_3}(M_3)\rightarrow \Omega^\bullet_{\CO_1}(M_1)\widehat\otimes\Omega^\bullet_{\CO_2}(M_2)\] of $\Psi$,
 we have another topological complex isomorphism
 \[{}^t(\Psi^{-1}):\ {}^t\left(\Omega^\bullet_{\CO_1}(M_1)\widehat\otimes\Omega^\bullet_{\CO_2}(M_2)\right)\rightarrow \mathrm D^{-\infty}_c(M_3;\Omega_{\CO_3}^{-\bullet}).
\]
By taking the composition of the previous two isomorphisms, we obtain a topological complex isomorphism
\be\label{defPhi}
\Phi={}^t(\Psi^{-1})\circ \tau:\quad \mathrm D^{-\infty}_c(M_1;\Omega_{\CO_1}^{-\bullet})\widehat\otimes_{\mathrm{i}}\,\mathrm D^{-\infty}_c(M_2;\Omega_{\CO_2}^{-\bullet})\rightarrow
\mathrm D^{-\infty}_c(M_3;\Omega_{\CO_3}^{-\bullet}).
\ee
This completes the proof of Theorem \ref{thm:tensorofcomplex}.
\qed
\subsection{Tensor products of  de Rham complexes II}
The main goal of this subsection is to prove the following theorem,
which  together with Theorem \ref{thm:tensorofcomplex} implies Theorem \ref{Introthm:tensorofcomplex}.

\begin{thmd}\label{thm:tensorofcscomplex} Let $(M_1,\CO_1)$ and $(M_2,\CO_2)$ be two  formal manifolds and let
$(M_3,\CO_3)$
be the product of them. Then we have the identifications
\begin{eqnarray*}
\mathrm D^{\infty}_c(M_3;\Omega_{\CO_3}^{-\bullet})&=&\mathrm D^{\infty}_c(M_1;\Omega _{\CO_1}^{-\bullet})\widetilde\otimes_{\mathrm{i}} \mathrm D^{\infty}_c(M_2;\Omega_{\CO_2}^{-\bullet})\\
&=&\mathrm D^{\infty}_c(M_1;\Omega _{\CO_1}^{-\bullet})\wh\otimes_{\mathrm{i}} \mathrm D^{\infty}_c(M_2;\Omega_{\CO_2}^{-\bullet})
\end{eqnarray*}
and
\begin{eqnarray*}
\mathrm C^{-\infty}(M_3;\Omega_{\CO_3}^{\bullet})&=&\mathrm C^{-\infty}(M_1;\Omega_{\CO_1}^{\bullet})\widetilde\otimes\mathrm C^{-\infty}(M_2;\Omega_{\CO_2}^{\bullet})\\
&=&\mathrm C^{-\infty}(M_1;\Omega_{\CO_1}^{\bullet})\widehat\otimes \mathrm C^{-\infty}(M_2;\Omega_{\CO_2}^{\bullet})
\end{eqnarray*}
of topological cochain complexes.
\end{thmd}

As a consequence of Theorems \ref{thm:tensorofcomplex} and \ref{thm:tensorofcscomplex}, we get the following generalization
of the Schwartz kernel theorems (see \eqref{eq:schwartzkernel2}-\eqref{eq:schwartzkernel4}) for the spaces of
compactly supported smooth densities, compactly supported distributions, and generalized functions.
 
\begin{cord}  We have the following LCS identifications:
\begin{eqnarray*}
    \mathrm D^{\infty}_c(M_3;\CO_3)&=&\mathrm D^{\infty}_c(M_1;\CO_1)\wt\otimes_{\mathrm{i}}\,
\mathrm D^{\infty}_c(M_2;\CO_2)\\
&=&\mathrm D^{\infty}_c(M_1;\CO_1)\wh\otimes_{\mathrm{i}}\,
\mathrm D^{\infty}_c(M_2;\CO_2),\\
\mathrm D^{-\infty}_c(M_3;\CO_3)&=&\mathrm D^{-\infty}_c(M_1;\CO_1)\wt\otimes_{\mathrm{i}}\,
\mathrm D^{-\infty}_c(M_2;\CO_2)\\
&=&\mathrm D^{-\infty}_c(M_1;\CO_1)\wh\otimes_{\mathrm{i}}\,
\mathrm D^{-\infty}_c(M_2;\CO_2),\\
\mathrm C^{-\infty}(M_3;\CO_3)&=&\mathrm C^{-\infty}(M_1;\CO_1)\wt\otimes\, \mathrm C^{-\infty}(M_2;\CO_2)\\
&=&\mathrm C^{-\infty}(M_1;\CO_1)\wh\otimes\,\mathrm C^{-\infty}(M_2;\CO_2).
\end{eqnarray*}
\end{cord}

In the rest of this subsection, we prove Theorem \ref{thm:tensorofcscomplex}. 
For $\eta_1\in \mathrm D^{-\infty}_c(V_1;\Omega_{\CO_1}^{r_1})$ and $\eta_2\in \mathrm D^{-\infty}_c(V_2;\Omega_{\CO_2}^{r_2})$, where $r_1,r_2\in \BN$ and $V_1\subset M_1$, $V_2\subset M_2$ are open subsets, we define   
\[\eta_1\boxtimes\eta_2:=\Phi(\eta_1\otimes\eta_2)\in \mathrm{D}_c^{-\infty}(V_1\times V_2;\Omega_{\CO_3}^{r_1+r_2})\] for convenience.  Here the map \[\Phi: \  \mathrm D^{-\infty}_c(V_1;\Omega_{\CO_1}^{-\bullet})\widehat\otimes_{\mathrm{i}}\,\mathrm D^{-\infty}_c(V_2;\Omega_{\CO_2}^{-\bullet})\rightarrow
\mathrm D^{-\infty}_c(V_1\times V_2;\Omega_{\CO_3}^{-\bullet})\] is the topological complex isomorphism defined  in \eqref{defPhi}.
For open subsets $V_1\subset U_1$ of $M_1$ and $V_2\subset U_2$ of $M_2$, it is easy to verify that 
\be\label{eq:rescommutesboxtimes}
\mathrm{ext}_{U_1,V_1}(\eta_1)\boxtimes  \mathrm{ext}_{U_2,V_2}(\eta_2)
=\mathrm{ext}_{U_1\times U_2,V_1\times V_2} (\eta_1\boxtimes \eta_2). 
\ee
The proof of Theorem \ref{thm:tensorofcscomplex} relies on  the following two results.

\begin{prpd}\label{prop:boxtimeswelldefined} Let $\eta_1\in \mathrm D^{\infty}_{K_1,s_1}(M_1;\Omega_{\CO_1}^{r_1})$ and $\eta_2\in \mathrm D^{\infty}_{K_2,s_2}(M_1;\Omega_{\CO_2}^{r_2})$, where $r_1,r_2\in \BN$, $(K_1,s_1)\in \mathcal{C}(M_1)$ and $(K_2,s_2)\in \mathcal{C}(M_2)$.  Then we have that 
\[\eta_1\boxtimes \eta_2\in \mathrm D^{\infty}_{K_1\times K_2,s_1+s_2}(M_3;\Omega_{\CO_3}^{r_1+r_2}).\]
\end{prpd}

In view of the isomorphism \eqref{defPhi},
Proposition \ref{prop:boxtimeswelldefined} implies that  there is an injective  linear map 
\be \label{eq:defcsdfiso}\begin{array}{rcl}
	\bigoplus\limits_{r_1,r_2\in \BN; r_1+r_2=r}\left(\mathrm D^{\infty}_{K_1,s_1}(M_1;\Omega_{\CO_1}^{r_1})\otimes_{\pi} \mathrm D^{\infty}_{K_2,s_2}(M_2;\Omega_{\CO_2}^{r_2})\right)&\rightarrow & \mathrm D^{\infty}_{K_1\times K_2,s_1+s_2}(M_3;\Omega_{\CO_3}^{r}),\\
	\{\eta_{r_1}\otimes \eta_{r_2}\}_{r_1,r_2\in \BN; r_1+r_2=r}&\mapsto & \sum\limits_{r_1,r_2\in \BN; r_1+r_2=r} \eta_{r_1}\boxtimes \eta_{r_2}
\end{array}
\ee for every $r\in \BN$, $(K_1,s_1)\in \mathcal{C}(M_1)$ and $(K_2,s_2)\in \mathcal{C}(M_2)$. Set 
\[
\oT_s:=\{(s_1,s_2)\in \BN\times \BN\mid s_1+s_2\le s\} \qquad (s\in\BN).
\]
As a subset of the product $\BN\times \BN$ of two totally ordered sets, $\oT_s$ is naturally a partially ordered set. 


\begin{prpd}\label{prop:csdfisoK} The injective linear  map \eqref{eq:defcsdfiso} is continuous and  induces a
topological linear isomorphism 
\be \label{eq:csdfisoK}
\mathrm{D}^{\infty,r}_{K_1,K_2,s}\rightarrow 
\mathrm D^{\infty}_{K_1\times K_2,s}(M_3;\Omega_{\CO_3}^r),
\ee
where $r,s\in \BN$ and
\be \label{eq:notationDK_1K_2sr}
\mathrm{D}^{\infty,r}_{K_1,K_2,s}:=
\varinjlim_{(s_1,s_2)\in \oT_s} \left(\bigoplus_{r_1,r_2\in \BN; r_1+r_2=r}
\left(\mathrm D^{\infty}_{K_1,s_1}(M_1;\Omega_{\CO_1}^{r_1})\widehat\otimes\mathrm D^{\infty}_{K_2,s_2}(M_2;\Omega_{\CO_2}^{r_2})\right)\right).
\ee
\end{prpd}

\vspace{3mm}

We begin to prove Proposition \ref{prop:boxtimeswelldefined}.
Let $n_1,n_2,k_1,k_2\in \BN$, and let $N_1\subset \R^{n_1}$, $N_2\subset \R^{n_2}$ be two open submanifolds.
Set \[n_3:=n_1+n_2,\ k_3:=k_1+k_2,\quad\text{and}\quad N_3:=N_1\times N_2\subset \R^{n_3}.\]

\begin{lemd} Assume that $M_1=N_1^{(k_1)}$ and $M_2=N_2^{(k_2)}$.
Then for every $(K_1,s_1)\in \CC(M_1)$ and $(K_2,s_2)\in \CC(M_2)$, we have that 
\be \label{eq:Kstilde=Kshat} \mathrm D^{\infty}_{K_1,s_1}(M_1;\CO_1)\widetilde\otimes \mathrm D^{\infty}_{K_2,s_2}(M_2;\CO_2)=  \mathrm D^{\infty}_{K_1,s_1}(M_1;\CO_1)\widehat\otimes  \mathrm D^{\infty}_{K_2,s_2}(M_2;\CO_2)
\ee as LCS.
Also, for every $s\in \BN$, we have an LCS identification
\be \label{eq:tensorofdessupinK} \varinjlim_{(s_1,s_2)\in \oT_s} \mathrm D^{\infty}_{K_1,s_1}(M_1;\CO_1)\widehat\otimes\,  \mathrm D^{\infty}_{K_2,s_2}(M_2;\CO_2)
= \mathrm D^{\infty}_{K_1\times K_2,s}(M_3;\CO_3).
\ee
\end{lemd}
\begin{proof} 
From \eqref{eq:schwartzkernel1}, it follows that \be\label{eq:schwartzkernel2'}\RC^\infty_{K_1}(N_1)\widetilde\otimes\RC^\infty_{K_2}
(N_2)=\RC^\infty_{K_1}(N_1)\widehat\otimes \RC^\infty_{K_2}(N_2)
=\RC^\infty_{K_1\times K_2}(N_1\times N_2).\ee  
The LCS identification \eqref{eq:Kstilde=Kshat} follows from \eqref{eq:schwartzkernel2'} and
\cite[(2.38)]{CSW2}. 
For the identification \eqref{eq:tensorofdessupinK}, by using  \eqref{eq:schwartzkernel2} and
 \cite[(2.38)]{CSW2}, we have that
\begin{eqnarray*}
&& \varinjlim_{(s_1,s_2)\in \oT_s} \mathrm D^{\infty}_{K_1,s_1}(M_1;\CO_1)\widehat\otimes\,  \mathrm D^{\infty}_{K_2,s_2}(M_2;\CO_2)\\
&=&\varinjlim_{(s_1,s_2)\in \oT_s} (\mathrm{D}^\infty_{K_1}(N_1)[y_1^*,y_2^*,\dots,y_{k_1}^*]_{\le s_1})\widehat\otimes
(\mathrm{D}^\infty_{K_2}(N_2)[y_1^*,y_2^*,\dots,y_{k_2}^*]_{\le s_2})\\
&=& (\mathrm{D}^\infty_{K_1}(N_1)\widehat\otimes\mathrm{D}^\infty_{K_2}(N_2))\otimes \left(\varinjlim_{(s_1,s_2)\in \oT_s}
 \C[y_1^*,y^*_2,\dots,y^*_{k_1}]_{\le s_1}\otimes
\C[y_1^*,y_2^*,\dots,y^*_{k_2}]_{\le s_2}\right)\\
&=& \mathrm{D}^\infty_{K_1\times K_2}(N_3)[y_1^*,y^*_2,\dots,y^*_{k_3}]_{\le s}
= \mathrm D^{\infty}_{K_1\times K_2,s}(M_3;\CO_3)
\end{eqnarray*} as LCS.
\end{proof}

\begin{lemd}\label{lem:boxtimes}
Proposition \ref{prop:boxtimeswelldefined} holds when $M_1=N_1^{(k_1)}$ and $M_2=N_2^{(k_2)}$.
\end{lemd}
\begin{proof}
By using Lemma \ref{lem:charOmegaonkr}, for $i=1,2$, we  assume  without loss of generality that 
\[\eta_i=\tau_i\, dx^*_{I_i} dy^*_{J_i}\in \mathrm D^{\infty}_{K_i,s_i}(M_i;\Omega_{\CO_i}^{r_i}), \]
where \[\tau_i\in \mathrm{D}_{K_i}^\infty(N_i)[y_1^*,y_2^*,\dots,y_{k_i}^*]_{\le s_i}\quad \text{and}\quad  (I_i,J_i)\in \Lambda_{n_i}^{n_i-t_i}\times \Lambda_{k_i}^{k_i-r_i+t_i}\] for some $t_i\in \BN$.
It is straightforward to check that 

\be \label{eq:comboxtimes}\eta_1\boxtimes \eta_2=(-1)^a(\tau_1\otimes \tau_2) dx^*_{(I_1,n_1+I_2)} dy^*_{(J_1,k_1+J_2)},
\ee
where \[a=t_2k_1+n_2r_1+n_2t_1+r_2k_1+n_1r_2+r_1t_2+t_1t_2,\] $\tau_1\otimes \tau_2$ is viewed as an element in $\mathrm{D}_{K_1\times K_2}^\infty(N_3)[y_1^*,y_2^*,\dots,y_{k_3}^*]_{\le s_1+s_2}$ via the
identification \eqref{eq:tensorofdessupinK}, and $(I_1,n_1+I_2)$, $(J_1,k_1+J_2)$ are defined as in \eqref{eq:defn+I}.
This completes the proof.
\end{proof}

\vspace{3mm}

\noindent\textbf{Proof of Proposition \ref{prop:boxtimeswelldefined}:} Let $\eta_i\in  \mathrm D^{\infty}_{K_i,s_i}(M_i;\Omega_{\CO_i}^{r_i})$ be as in Proposition \ref{prop:boxtimeswelldefined}. For each $i=1,2$, take a finite family $\{U_{i,\alpha}\}_{\alpha\in \Gamma_i}$ of  charts of $M_i$ that covers $K_i$. 
Let $\{f_{i,\alpha}\}_{\alpha\in\Gamma_{i}\sqcup\{0\}}$ be  
a partition of unity of $M_i$ subordinate to the open cover $\{U_{i,\alpha}\}_{\alpha\in \Gamma_i\sqcup\{0\}}$ (see \cite[Proposition 2.3]{CSW1}), where $U_{i,0}:=M_i\setminus K_i$. Set $K_{i,\alpha}:=\mathrm{supp}f_{i,\alpha}\cap K_i$ for all $\alpha\in \Gamma_i$.

It follows from \cite[Lemma 2.14]{CSW2} that
\[ \eta_i=\sum_{\alpha\in \Gamma_i} \eta_i\circ f_{i,\alpha}
=\sum_{\alpha\in \Gamma_i} \mathrm{ext}_{M_i,U_{i,\alpha}} (\eta_{i,\alpha}),
\]
where $\eta_{i,\alpha}:=(\eta_i\circ f_{i,\alpha})|_{U_{i,\alpha}}\in \mathrm D^{\infty}_{K_{i,\alpha},s_i}(U_{i,\alpha};\Omega_{\CO_i}^{r_i})$.
Meanwhile,  Lemma \ref{lem:boxtimes} implies that
\[\eta_{1,\alpha}\boxtimes \eta_{2,\beta}\in \mathrm D^{\infty}_{K_{1,\alpha}\times K_{2,\beta},s_1+s_2}(U_{1,\alpha }\times U_{2,\beta};\Omega_{\CO_3}^{r_1+r_2})
\quad (\alpha\in \Gamma_1, \beta\in \Gamma_2).
\]
These, together with \eqref{eq:rescommutesboxtimes},  imply that
\begin{eqnarray*}
&&\eta_1\boxtimes \eta_2=\left(\sum_{\alpha\in \Gamma_1} \mathrm{ext}_{M_1,U_{1,\alpha}} (\eta_{1,\alpha})\right)
\boxtimes \left(\sum_{\beta\in \Gamma_2} \mathrm{ext}_{M_2,U_{2,\beta}} (\eta_{2,\beta})\right)\\
&=& \sum_{(\alpha,\beta)\in \Gamma_1\times \Gamma_2}
\mathrm{ext}_{M_1\times M_2,U_{1,\alpha}\times U_{2,\beta}} (\eta_{1,\alpha}\boxtimes \eta_{2,\beta})
\in  \mathrm D^{\infty}_{K_1\times K_2,s_1+s_2}(M_3;\Omega_{\CO_3}^{r_1+r_2}).
\end{eqnarray*}
This finishes the proof of Proposition \ref{prop:boxtimeswelldefined}.
 \qed
\vspace{3mm}

Now we turn to prove Proposition \ref{prop:csdfisoK}.

\begin{lemd}\label{lem:csdfisopre} Proposition \ref{prop:csdfisoK} holds when $M_1=N_1^{(k_1)}$ and $M_2=N_2^{(k_2)}$.
\end{lemd}
\begin{proof} The assertions follow from the equality \eqref{eq:comboxtimes},
the topological linear isomorphism \eqref{eq:topiosondnck-ru2}, the
LCS identification \eqref{eq:tensorofdessupinK}, and the following bijection
\begin{eqnarray*}
\bigsqcup_{r_1,r_2\in \BN;r_1+r_2=r}\Lambda_{n_1,k_1}^{n_1+k_1-r_1}\times \Lambda_{n_2,k_2}^{n_2+k_2-r_2}&\rightarrow& \Lambda_{n_3,k_3}^{n_3+k_3-r},\\
((I_1,J_1),(I_2,J_2))&\mapsto & ((I_1,n_1+I_2),(J_1,k_1+J_2)),
\end{eqnarray*}
where $(I_1,n_1+I_2)$ and $(J_1,k_1+J_2)$ are defined as in \eqref{eq:defn+I}.
\end{proof}

With  the  notations used in the proof of Proposition \ref{prop:boxtimeswelldefined},
by applying \eqref{eq:rescommutesboxtimes},  
one concludes that the diagram
\be\label{eq:diapre12.18}
 \begin{CD}
             \bigoplus\limits_{\alpha\in \Gamma_1}\mathrm D^{\infty,r_1}_{K_{1,\alpha},s_1}\otimes_{\pi}
            \bigoplus\limits_{\beta\in \Gamma_2}\mathrm  D^{\infty,r_2}_{K_{2,\beta},s_2}@>  >>  \bigoplus\limits_{(\alpha,\beta)\in \Gamma_1\times \Gamma_2}
           \mathrm D^{\infty,r}_{K_{1,\alpha}\times K_{2,\beta},s}\\
          @VV(\oplus\, \mathrm{ext}_{M_1,U_{1,\alpha}})\otimes(\oplus\, \mathrm{ext}_{M_2,U_{2,\beta}})    V
                  @V V \oplus\,\mathrm{ext}_{M_3,U_{1,\alpha}\times U_{2,\beta}} V\\
           \mathrm D^{\infty}_{K_1,s_1}(M_1;\Omega_{\CO_1}^{r_1})\otimes_{\pi}\mathrm D^{\infty}_{K_2,s_2}(M_2;\Omega_{\CO_2}^{r_2}) @ >> >
           \mathrm D^{\infty}_{K_1\times K_2,s}(M_3;\Omega_{\CO_3}^r)
  \end{CD}
\ee commutes. Here $r=r_1+r_2$, $s=s_1+s_2$, 
\begin{eqnarray}
\label{eq:notionindiag1}\mathrm D^{\infty,r_1}_{K_{1,\alpha},s_1}&:=&\mathrm D^{\infty}_{K_{1,\alpha},s_1}(U_{1,\alpha};\Omega_{\CO_1}^{r_1}),\\
\label{eq:notionindiag2}\mathrm  D^{\infty,r_2}_{K_{2,\beta},s_2}&:=&\mathrm  D^{\infty}_{K_{2,\beta},s_2}(U_{2,\beta};\Omega_{\CO_2}^{r_2}),\\
 \label{eq:notionindiag3} \mathrm D^{\infty,r}_{K_{1,\alpha}\times K_{2,\beta},s}&:=&\mathrm D^{\infty}_{K_{1,\alpha}\times K_{2,\beta},s}(U_{1,\alpha}\times U_{2,\beta};\Omega_{\CO_3}^{r}),
\end{eqnarray}
 and the top horizontal arrow in the diagram is the composition of the canonical topological linear isomorphism
\[  
\left(\bigoplus\limits_{\alpha\in \Gamma_1}\mathrm D^{\infty,r_1}_{K_{1,\alpha},s_1}\right)\otimes_{\pi}
            \left(\bigoplus\limits_{\beta\in \Gamma_2}\mathrm  D^{\infty,r_2}_{K_{2,\beta},s_2}\right)
             \rightarrow  \bigoplus\limits_{(\alpha,\beta)\in \Gamma_1\times \Gamma_2}
          \left( \mathrm D^{\infty,r_1}_{K_{1,\alpha},s_1}\otimes_{\pi}
          \mathrm  D^{\infty,r_2}_{K_{2,\beta},s_2}\right)
             \]
with the linear topological embedding (see Lemma \ref{lem:csdfisopre})
\[\bigoplus\limits_{(\alpha,\beta)\in \Gamma_1\times \Gamma_2}
           \left(\mathrm D^{\infty,r_1}_{K_{1,\alpha},s_1}\otimes_{\pi}
          \mathrm  D^{\infty,r_2}_{K_{2,\beta},s_2}\right)
             \rightarrow \bigoplus\limits_{(\alpha,\beta)\in \Gamma_1\times \Gamma_2}
           \mathrm D^{\infty,r}_{K_{1,\alpha}\times K_{2,\beta},s}.
             \]
\begin{lemd}\label{lem:pre12.18}
 The right vertical arrow in \eqref{eq:diapre12.18} is a continuous, open and surjective linear map.    
\end{lemd}
\begin{proof}
    Note  that
$\{U_{1,\alpha}\times U_{2,\beta}\}_{(\alpha,\beta)\in \Gamma_1\times \Gamma_2}$ is a family  of  open subsets  of $M_3$  that covers $K_1\times K_2$, 
$\{f_{1,\alpha}\otimes f_{2,\beta}\}_{(\alpha,\beta)\in \Gamma_1\times \Gamma_2}$ is a family of formal functions on $M_3$ such that
\[\mathrm{supp}\,(f_{1,\alpha}\otimes f_{2,\beta}) \subset U_{1,\alpha}\times U_{2,\beta}\quad\text{for all $(\alpha,\beta)\in \Gamma_1\times \Gamma_2$}\] and 
\[\sum_{(\alpha,\beta)\in\Gamma_1\times \Gamma_2} (f_{1,\alpha}\otimes f_{2,\beta})|_{K_1\times K_2}=1.\]
We also have that 
\begin{eqnarray*}
K_{1,\alpha}\times K_{2,\beta}&=&(\mathrm{supp}f_{1,\alpha}\times \mathrm{supp}f_{2,\beta})\cap (K_1\times K_2)\\
&=&(\mathrm{supp}(f_{1,\alpha}\otimes f_{2,\beta}))\cap (K_1\times K_2).\end{eqnarray*}
The lemma then follows from \cite[Lemma 2.14]{CSW2}.
\end{proof}
\begin{lemd}\label{lem:csdfpre2} The injective linear map  \eqref{eq:defcsdfiso} is continuous.
\end{lemd}
\begin{proof} 
\cite[Lemma 2.14]{CSW2} and \cite[Proposition 43.9]{Tr} imply that the left vertical arrow in \eqref{eq:diapre12.18} is a continuous, open and surjective linear map. 
Thus the bottom horizontal arrow in \eqref{eq:diapre12.18} is continuous, and the lemma then follows. 
\end{proof}

\vspace{3mm}

\noindent\textbf{Proof of Proposition \ref{prop:csdfisoK}:} 
In view of Lemma \ref{lem:csdfpre2}, the map \eqref{eq:defcsdfiso} yields an injective continuous linear map 
\[
\bigoplus_{r_1,r_2\in \BN; r_1+r_2=r}
\left(\mathrm D^{\infty}_{K_1,s_1}(M_1;\Omega_{\CO_1}^{r_1})\widehat\otimes\mathrm D^{\infty}_{K_2,s_2}(M_2;\Omega_{\CO_2}^{r_2})\right)\rightarrow D^{\infty}_{K_1\times K_2,s_1+s_2}(M_3;\Omega_{\CO_3}^r).
\]
Then we have an  injective continuous linear map
\[
  \mathrm{D}_{K_1,K_2,s}^{\infty,r}\rightarrow \mathrm D^{\infty}_{K_1\times K_2,s}(M_3;\Omega_{\CO_3}^r)
\] 
as  in \eqref{eq:csdfisoK}.

On the other hand,   
the commutative diagram \eqref{eq:diapre12.18}
 can be extended to the following commutative diagram (see \eqref{eq:notationDK_1K_2sr} and \eqref{eq:notionindiag1}-\eqref{eq:notionindiag3} for the notations):
\[
 \begin{CD}
             \varinjlim\limits_{(s_1,s_s)\in \oT_s}\left(\bigoplus\limits_{r_1,r_2\in \BN;r_1+r_2=r}\left(\bigoplus\limits_{\alpha\in \Gamma_1}\mathrm D^{\infty,r_1}_{K_{1,\alpha},s_1}\wh\otimes
            \bigoplus\limits_{\beta\in \Gamma_2}\mathrm  D^{\infty,r_2}_{K_{2,\beta},s_2}\right)\right)
             @>  >>    \bigoplus\limits_{(\alpha,\beta)\in \Gamma_1\times \Gamma_2}
           \mathrm D^{\infty,r}_{K_{1,\alpha}\times K_{2,\beta},s}\\
            @V V V
                  @V V V\\
            \mathrm{D}_{K_1,K_2,s}^{\infty,r}  @ > \eqref{eq:csdfisoK}> >
            \mathrm D^{\infty}_{K_1\times K_2,s}(M_3;\Omega_{\CO_3}^r).
  \end{CD}
\]

From Lemma \ref{lem:csdfisopre}, it follows that the top horizontal arrow is a topological linear isomorphism.
Since $ \mathrm D^{\infty,r_1}_{K_{1,\alpha},s_1} $ and $\mathrm D^{\infty,r_2}_{K_{2,\beta},s_2} $ are Fr\'echet spaces
(see \cite[Proposition 2.20]{CSW2}),
 the left vertical arrow is   open and surjective  by \cite[Proposition 43.9]{Tr}. Meanwhile,  the right vertical arrow is open and surjective by Lemma \ref{lem:pre12.18}.
These imply that the bottom horizontal arrow is open and surjective. Then it is a topological linear isomorphism, as required. \qed
\vspace{3mm}

Now we are ready to prove Theorem \ref{thm:tensorofcscomplex}. 
To begin with, we have the following result.

\begin{lemd}\label{lem:Omegactilde=hat} For $r_1,r_2\in \BN$, we have that 
\[\mathrm D^{\infty}_c(M_1;\Omega_{\CO_1}^{r_1})\widetilde\otimes_{\mathrm{i}}\, \mathrm D^{\infty}_c(M_2;\Omega_{\CO_2}^{r_2})
=\mathrm D^{\infty}_c(M_1;\Omega_{\CO_1}^{r_1})\wh\otimes_{\mathrm{i}}\, \mathrm D^{\infty}_c(M_2;\Omega_{\CO_2}^{r_2})\] as LCS.
\end{lemd}
\begin{proof}
Let $r_1,r_2\in \BN$, $(K_1,s_1)\in \mathcal{C}(M_1)$ and $(K_2,s_2)\in \mathcal{C}(M_2)$.
The left vertical arrow of the commutative diagram \eqref{eq:diapre12.18} induces the following commutative
diagram (see \eqref{eq:notionindiag1} and \eqref{eq:notionindiag2} for the notations):
\[\begin{CD}
\bigoplus\limits_{\alpha\in \Gamma_1}\mathrm D^{\infty,r_1}_{K_{1,\alpha},s_1} \wt\otimes
  \bigoplus\limits_{\beta\in \Gamma_2}\mathrm  D^{\infty,r_2}_{K_{2,\beta},s_2}@> >>
            \mathrm D^{\infty}_{K_{1},s_1}(M_1;\Omega_{\CO_1}^{r_1})\wt \otimes\mathrm  D^{\infty}_{K_{2},s_2}(M_2;\Omega_{\CO_2}^{r_2})\\
               @V V V
                  @V V V\\
 \bigoplus\limits_{\alpha\in \Gamma_1}\mathrm D^{\infty,r_1}_{K_{1,\alpha},s_1} \wh\otimes
  \bigoplus\limits_{\beta\in \Gamma_2}\mathrm  D^{\infty,r_2}_{K_{2,\beta},s_2}@> >>
    \mathrm D^{\infty}_{K_{1},s_1}(M_1;\Omega_{\CO_1}^{r_1})\wh \otimes\mathrm  D^{\infty}_{K_{2},s_2}(M_2;\Omega_{\CO_2}^{r_2}).
              \end{CD}
              \]

By 
\eqref{eq:Kstilde=Kshat} and \eqref{eq:topiosondnck-ru2}, we have that
\begin{eqnarray*}  \mathrm D^{\infty,r_1}_{K_{1,\alpha},s_1}\wt\otimes
\mathrm  D^{\infty,r_2}_{K_{2,\beta},s_2}
=\mathrm D^{\infty,r_1}_{K_{1,\alpha},s_1}\wh\otimes
\mathrm  D^{\infty,r_2}_{K_{2,\beta},s_2}
\end{eqnarray*}
as LCS.
This implies that the left vertical arrow is a topological linear isomorphism.
Note that the bottom horizontal arrow is a continuous, open and surjective linear map by \cite[Proposition 43.9]{Tr}, which forces that the right vertical arrow is surjective and hence
\be\begin{split}\label{eq:DK1s1=DK2s2}
&\mathrm D^{\infty}_{K_{1},s_1}(M_1;\Omega_{\CO_1}^{r_1})\wt\otimes
\mathrm  D^{\infty}_{K_{2},s_2}(M_2;\Omega_{\CO_2}^{r_2})\\
=\quad& \mathrm D^{\infty}_{K_{1},s_1}(M_{1};\Omega_{\CO_1}^{r_1})\wh\otimes
\mathrm  D^{\infty}_{K_{2},s_2}(M_{2};\Omega_{\CO_2}^{r_2}).\end{split}\ee

Consider  the following commutative diagram:
\[\begin{CD}
\varinjlim\limits_{((K_1,s_1),(K_2,s_2))}\left(\mathrm D^{\infty,r_1}_{K_{1},s_1}\widetilde\otimes_{\mathrm{i}}\,\mathrm  D^{\infty,r_2}_{K_{2},s_2}\right)@> >>
            \varinjlim\limits_{((K_1,s_1),(K_2,s_2))}\left(\mathrm D^{\infty,r_1}_{K_{1},s_1}\wh\otimes_{\mathrm{i}}\,\mathrm  D^{\infty,r_2}_{K_{2},s_2}\right)\\
               @V V V
                  @V V V\\
 \mathrm D^{\infty}_c(M_1;\Omega_{\CO_1}^{r_1})\widetilde\otimes_{\mathrm{i}}\, \mathrm D^{\infty}_c(M_2;\Omega_{\CO_2}^{r_2})@> >>
   \mathrm D^{\infty}_c(M_1;\Omega_{\CO_1}^{r_1})\wh\otimes_{\mathrm{i}}\, \mathrm D^{\infty}_c(M_2;\Omega_{\CO_2}^{r_2}),
              \end{CD}
              \] where $((K_1,s_1),(K_2,s_2))$ runs over $\mathcal{C}(M_1)\times \mathcal{C}(M_2)$,
    \[\mathrm D^{\infty,r_1}_{K_{1},s_1}:=\mathrm D^{\infty}_{K_1,s_1}(M_1;\Omega_{\CO_1}^{r_1}),\quad \text{and} \quad   
    \mathrm  D^{\infty,r_2}_{K_{2},s_2}:=
\mathrm  D^{\infty}_{K_{2},s_2}(M_2;\Omega_{\CO_2}^{r_2}).\]
By  \cite[(2.12) and Proposition 2.20]{CSW2}, it follows from  \cite[Lemma A.14]{CSW2} that the right vertical arrow is a topological linear isomorphism. 
The top horizontal arrow is also a topological linear  isomorphism by 
\eqref{eq:DK1s1=DK2s2}. 
Thus the bottom horizontal arrow, which is a linear topological embedding, 
must be an isomorphism, as required. 
\end{proof}

By applying \cite[Proposition 3.5]{CSW2}, we also have the following result.

\begin{lemd}\label{lem:tOmegactilde=hat} For $r_1,r_2\in \BN$, we have that
\[\mathrm C^{-\infty}(M_1;\Omega_{\CO_1}^{r_1})\widetilde\otimes\, \mathrm C^{-\infty}(M_2;\Omega_{\CO_2}^{r_2})
=\mathrm C^{-\infty}(M_1;\Omega_{\CO_1}^{r_1})\widehat\otimes\, \mathrm C^{-\infty}(M_2;\Omega_{\CO_2}^{r_2})\]
as LCS.
\end{lemd}

\vspace{3mm}

\noindent\textbf{Proof of Theorem \ref{thm:tensorofcscomplex}:}
Set \begin{eqnarray*}
    &&\mathcal{C}'(M_3):=\\&&\{(K_1\times K_2,s)\mid s\in \BN, K_1\subset M_1 \text{ and }K_2\subset M_2\ \text{are compact subsets}
\}\subset \mathcal{C}(M_3).
\end{eqnarray*}

Let $r\in \BN$.
By using \cite[(2.12) and Lemma A.14]{CSW2}, and
Proposition  \ref{prop:csdfisoK}, we have the following LCS identifications:
\begin{eqnarray*}
&&\bigoplus_{r_1,r_2\in \BN;r_1+r_2=r}\mathrm D^{\infty}_c(M_1;\Omega_{\CO_1}^{r_1})\wh\otimes_{\mathrm{i}}\, \mathrm D^{\infty}_c(M_2;\Omega_{\CO_2}^{r_2})\\
&=&\bigoplus_{r_1,r_2\in \BN;r_1+r_2=r}\left(\varinjlim_{(K_1,s_1)\in \mathcal{C}(M_1)}\mathrm D^{\infty}_{K_1,s_1}(M_1;\Omega_{\CO_1}^{r_1})
\widehat\otimes_{\mathrm{i}}\varinjlim_{(K_2,s_2)\in \mathcal{C}(M_2)}\mathrm D^{\infty}_{K_2,s_2}(M_2;\Omega_{\CO_2}^{r_2})\right)\\
&=&\bigoplus_{r_1,r_2\in \BN;r_1+r_2=r}\left(\varinjlim_{((K_1,s_1),(K_2,s_2))\in \mathcal{C}(M_1)\times \mathcal{C}(M_2)}
\mathrm D^{\infty}_{K_1,s_1}(M_1;\Omega_{\CO_1}^{r_1})\widehat\otimes_{\mathrm{i}}\mathrm D^{\infty}_{K_2,s_2}(M_2;\Omega_{\CO_2}^{r_2})\right)\\
&=&\varinjlim_{((K_1,s_1),(K_2,s_2))\in \mathcal{C}(M_1)\times \mathcal{C}(M_2)}\left(\bigoplus_{r_1,r_2\in \BN;r_1+r_2=r}
\mathrm D^{\infty}_{K_1,s_1}(M_1;\Omega_{\CO_1}^{r_1})\widehat\otimes_{\mathrm{i}}\mathrm D^{\infty}_{K_2,s_2}(M_2;\Omega_{\CO_2}^{r_2})\right)\\
&=& \varinjlim_{(K_1\times K_2,s)\in \mathcal{C}'(M_3)} \left( \varinjlim_{(s_1,s_2)\in \oT_s}\left(\bigoplus_{r_1,r_2\in \BN;r_1+r_2=r}
\mathrm D^{\infty}_{K_1,s_1}(M_1;\Omega_{\CO_1}^{r_1})\widehat\otimes_{\mathrm{i}}\mathrm D^{\infty}_{K_2,s_2}(M_2;\Omega_{\CO_2}^{r_2})\right)\right)\\
&=&\varinjlim_{(K_1\times K_2,s)\in \mathcal{C}'(M_3)} \mathrm D^{\infty}_{K_1\times K_2,s}(M_3;\Omega_{\CO_3}^{r})\\
&=&\varinjlim_{(K,s)\in \mathcal{C}(M_3)} \mathrm D^{\infty}_{K,s}(M_3;\Omega_{\CO_3}^{r})= \mathrm D^{\infty}_{c}(M_3;\Omega_{\CO_3}^{r}).
\end{eqnarray*}
Furthermore, by taking the strong dual,  it follows from \cite[Lemma A.14]{CSW2} that 
\[\bigoplus_{r_1,r_2\in \BN;r_1+r_2=r}\mathrm C^{-\infty}(M_1;\Omega_{\CO_1}^{r_1})\widehat\otimes\, \mathrm C^{-\infty}(M_2;\Omega_{\CO_2}^{r_2})
=\mathrm C^{-\infty}(M_3;\Omega_{\CO_3}^{r}).
\]
These, together with Lemmas \ref{lem:Omegactilde=hat} and  \ref{lem:tOmegactilde=hat}, finish the proof of Theorem \ref{thm:tensorofcscomplex}.

\qed

\section{Poincar\'e's lemma in the formal manifold setting}\label{sec:poin}
In this section, we generalize
the classical Poincar\'e's lemma of de Rham complexes 
on smooth manifolds to the setting of formal manifolds. 
As before, let  $E$ be a quasi-complete  LCS
throughout this section.


\subsection{Poincar\'e's lemma I} In this subsection, we prove the Poincar\'e's lemma for the de Rham complexes of
$(M,\CO)$ with coefficients {in} formal functions  or compactly supported formal distributions. 

Let $E_M(M)$ be the space of  $E$-valued locally constant functions over $M$, equipped with the point-wise convergence topology. 
Set 
\[
E_{M,\mathrm{f}}(M):=(\C_M(M))'\widetilde\otimes_{\mathrm{i}}\,E.
\]
Then we have the  LCS identifications
\[
E_M(M)=\prod_{Z\in \pi_0(M)}E\]and\[
E_{M,\mathrm{f}}(M)=\left(\bigoplus_{Z\in \pi_0(M)}\C\right)\widetilde\otimes_{\mathrm{i}}\,E=\bigoplus_{Z\in \pi_0(M)} E.
\]

Note that (see \cite[Lemma 6.6]{CSW1})
\[
\Omega^0_{\CO}(M;E)=\CO(M)\widetilde{\otimes}E=
\left(\prod_{Z\in \pi_0(M)}\CO(N)\right)\widetilde{\otimes}E=\prod_{Z\in \pi_0(M)}\left(\CO(N)\widetilde\otimes E\right).
\]
We define a canonical continuous linear map
\begin{eqnarray}\label{eq:varepsilon}
\varepsilon:\ E_M(M)=\prod_{Z\in \pi_0(M)}E&\rightarrow &\Omega^0_{\CO}(M;E)=\prod_{Z\in \pi_0(M)}\left(\CO(Z)\widetilde\otimes E\right),\\  \{v_Z\}_{Z\in \pi_0(M)}&\mapsto& \{1_Z\otimes v _Z\}_{Z\in \pi_0(M)},\notag
\end{eqnarray} where $1_Z$ is the identity element in $\CO(Z)$. 
Taking the  transpose of $\varepsilon: \,\BC_M(M)\rightarrow \Omega^0_{\CO}(M)$, we have a continuous linear map \be \label{eq:zataasamap}\zeta:\, {\mathrm D}^{-\infty}_c(M;\Omega_{\CO}^{0})\rightarrow \BC_{M,\mathrm{f}}(M).\ee 
Furthermore, by taking the quasi-completed inductive tensor product with $E$, we obtain a continuous linear map
\be \label{eq:zeta}\zeta: \,  {\mathrm D}^{-\infty}_c(M;\Omega_{\CO}^{0})\wt\otimes_{\mathrm{i}}E\rightarrow E_{M,\mathrm{f}}(M).\ee

 Note that the composition map \[d\circ\varepsilon:\ E_M(M)\rightarrow \Omega^0_\CO(M;E)\rightarrow \Omega^1_{\CO}(M;E)\] is zero, and so is  the composition map 
 \[ 
 \zeta\circ d:\
 {\mathrm D}^{-\infty}_c(M;\Omega_{\CO}^{1})\wt\otimes_{\mathrm{i}}E
 \rightarrow {\mathrm D}^{-\infty}_c(M;\Omega_{\CO}^{0})\wt\otimes_{\mathrm{i}}E\rightarrow E_{M,\mathrm{f}}(M).
 \]
Then as a slight modification of the de Rham complex \eqref{eq:Omegabullet}, we have an augmented de Rham complex
\be\label{eq:ardeRham1}\cdots\rightarrow 0\rightarrow E_M(M)\xrightarrow{\varepsilon} \Omega^0_{\CO}(M;E)\xrightarrow{d} \Omega^1_{\CO}(M;E)\xrightarrow{d} \cdots. \ee
  Similarly,  as a slight modification of the de Rham complex \eqref{eq:D-inftyc},   we have  a  coaugmented de Rham complex
    \be \label{eq:coardeRham1}\cdots \xrightarrow{d}
{\mathrm D}^{-\infty}_c(M;\Omega_{\CO}^{1})\wt\otimes_{\mathrm{i}}E\xrightarrow{d}{\mathrm D}^{-\infty}_c(M;\Omega_{\CO}^{0})\wt\otimes_{\mathrm{i}}E\xrightarrow{\zeta}E_{M,\mathrm{f}}(M)\rightarrow0\rightarrow  \cdots.
\ee

 The main goal of this subsection is to prove the following theorem, which implies Theorem \ref{pf} provided that $E=\C$. 

\begin{thmd}\label{thm:poinlem1}
    
    Suppose that $M=N^{(k)}$ for some contractible smooth manifold $N$ and $k\in \BN$. 
    Then   the topological cochain complexes
   \eqref{eq:ardeRham1}
    and 
   \eqref{eq:coardeRham1} are 
   both strongly exact.
\end{thmd}

We extend the maps 
\[\varepsilon: \,E_M(M)\rightarrow \Omega^0_{\CO}(M;E)\quad \text{and}\quad \zeta:\, {\mathrm D}^{-\infty}_c(M;\Omega_{\CO}^{0})\wt\otimes_{\mathrm{i}}E\rightarrow E_{M,\mathrm{f}}(M)\] to  continuous complex maps 
\be \label{eq:commapvarepsilon}
\varepsilon:\quad \iota^\bullet(E_M(M))\rightarrow \Omega_{\CO}^{\bullet}(M;E)
\ee 
and \be\label{eq:commapeta}
\zeta:\quad {\mathrm D}^{-\infty}_c(M;\Omega_{\CO}^{-\bullet})\wt\otimes_{\mathrm{i}}\iota ^{\bullet}(E)\rightarrow \iota^\bullet(E_{M,\mathrm{f}}(M)),
\ee
respectively.
Here 
$\iota^{\bullet}(E_M(M))$ and $\iota^\bullet(E_{M,\mathrm{f}}(M))$ are as in \eqref{eq:iotaE}.

Using Lemmas \ref{lem:tophomimplestrong1} and \ref{lem:tophomimplestrong2},  Theorem \ref{thm:poinlem1} 
is implied by the following result. 

\begin{thmd}\label{poin}
Suppose that $M=N^{(k)}$ for some contractible smooth manifold $N$ and $k\in \BN$.
 Then the continuous complex maps \eqref{eq:commapvarepsilon} and \eqref{eq:commapeta}
 are both topological homotopy equivalences.
\end{thmd}

We prove Theorem \ref{poin} in the rest part of this subsection.
Suppose that  $M$ is as in Theorem \ref{poin}. In this case, we have that $E_M(M)=E$. 
Then the continuous complex map  \eqref{eq:commapvarepsilon} can  be written as \[\varepsilon:\quad \iota^\bullet(E)\rightarrow \Omega_{\CO}^{\bullet}(M;E),
\] with 
the continuous linear map 
\[ \varepsilon:\,E_M(M)=E\rightarrow \Omega^0_{\CO}(M;E)=\CO(M)\widetilde\otimes E,\quad v \mapsto 1\otimes v.\] 
Similarly, the continuous complex map \eqref{eq:commapeta} can  be written as
\[\zeta:\quad {\mathrm D}^{-\infty}_c(M;\Omega_{\CO}^{-\bullet})\wt\otimes_{\mathrm{i}}\iota^{\bullet}(E)\rightarrow \iota^\bullet(E),\]  with the continuous linear map 
\begin{eqnarray}\label{eq:zetaspecial}
 \zeta:\ {\mathrm D}^{-\infty}_c(M;\Omega_{\CO}^{0})\wt\otimes_{\mathrm{i}}E=(\CO(M))'\,\widetilde\otimes_{\mathrm{i}} E&\rightarrow& E_{M,\mathrm{f}}(M)=E,\\ \eta \otimes v &\mapsto &\eta (1)v \notag\end{eqnarray}
  where $1\in \CO(M)$ is the identity element.

The following result is a form of Poincar\'{e}'s lemma (see \cite[\S\,4.3]{Man}).

\begin{lemd}\label{putong}
The complex map  \eqref{eq:commapvarepsilon} is a topological homotopy equivalence when $M=N^{(0)}$ and $E=\C$.
\end{lemd}


\begin{lemd}\label{dandian}
The complex map  \eqref{eq:commapvarepsilon} is a topological homotopy equivalence when $E=\C$ and  $M=(\R^0)^{(k)}$ for some $k\in\BN$.
\end{lemd}
\begin{proof} This is well-known. We provide a proof for the convenience of the reader. The complex map \eqref{eq:commapvarepsilon} is the identity map of $\iota^{\bullet}(\BC)$ when $k=0$. Now we assume that $k=1$.
In this case, by using Example \ref{ex:DF}, the complex $\Omega_{\CO}^{\bullet}(M)$ is as follows:
\[\cdots\rightarrow 0\rightarrow \C[[y]]\xrightarrow{d} \C[[y]]dy\rightarrow 0\rightarrow \cdots,\]
where the coboundary map $d$ is given by 
\[d: \C[[y]]\rightarrow \C[[y]]dy,\quad \sum_{i\in \BN} c_i y^i\mapsto \left(\sum_{i\in \BN\setminus \{0\}} i c_i y^{i-1}\right)dy.\]
Let $g:\Omega_{\CO}^{\bullet }(M)\rightarrow \iota^\bullet(\C)$ be the complex map determined by
\[g:\C[[y]]\rightarrow \C,\quad \sum_{i\in \BN} c_i y^i\mapsto c_0.\]
It is clear that
\[g\circ \varepsilon=\mathrm{id}_{\iota^{\bullet}(\BC)}:\ \iota^\bullet(\BC)\rightarrow \iota^\bullet(\C).\]

On the other hand, let $h:\Omega_{\CO}^{\bullet}(M)\rightarrow\Omega_{\CO}^{\bullet}(M)$ be the continuous map of degree $-1$ determined by
\[h:\BC[[y]]dy\rightarrow\BC[[y]],\quad
		\sum_{i\in\BN }c_iy^{i}dy \mapsto \sum_{i\in\BN }\frac{c_i}{i+1}y^{i+1}.\]
Then we have that
 \[h\circ d+d\circ h=\mathrm{id}-\varepsilon\circ g:\ \Omega_{\CO}^{\bullet}(M)\rightarrow\Omega_{\CO}^{\bullet}(M).\]
This implies that $\varepsilon$ is a  topological homotopy equivalence.

For the general case, since $(\R^0)^{(k)}$ is the product of $k$-copies of $(\R^0)^{(1)}$, it follows from Theorem \ref{propro}
that
\[\Omega_{\CO_{\R^0}^{(k)}}^{\bullet}(\BR^0)=\underbrace{\Omega_{\CO_{\R^0}^{(1)}}^{\bullet}(\BR^{0})\widetilde\otimes  \Omega_{\CO_{\R^0}^{(1)}}^{\bullet}(\BR^{0})\widetilde\otimes  \cdots\widetilde\otimes\Omega_{\CO_{\R^0}^{(1)}}^{\bullet}(\BR^{0})}_k.\]
Under this identification, the complex map $\varepsilon: \iota^\bullet(\C)\rightarrow \Omega_{\CO_{\R^0}^{(k)}}^{\bullet}(\BR^0)$
coincides with the following tensor product of complex maps: 
\[
\underbrace{\varepsilon\otimes \cdots\otimes \varepsilon}_k:\
\underbrace{\iota^\bullet(\C)\widetilde\otimes
 \cdots\widetilde\otimes\iota^\bullet(\C)}_k
\rightarrow
\underbrace{\Omega_{\CO_{\R^0}^{(1)}}^{\bullet}(\BR^{0})\widetilde
\otimes  \cdots\widetilde\otimes\Omega_{\CO_{\R^0}^{(1)}}^{\bullet}(\BR^{0})}_k.
\]
The assertion then follows from  Lemma \ref{prop:tensorhomoequi}.
\end{proof}

\begin{lemd}\label{lem:generalpo=C}
The complex maps  \eqref{eq:commapvarepsilon} and \eqref{eq:commapeta} are  topological homotopy equivalences when  $E=\C$.
\end{lemd}
\begin{proof}
As
\[M=N^{(k)}=N^{(0)}\times (\R^0)^{(k)},\]
it follows from Theorem \ref{propro}  that the complex map $\varepsilon:\iota^\bullet(\C)\rightarrow \Omega_{\CO_{N}^{(k)}}^{\bullet}(N)$ coincides with the tensor product
\[\varepsilon\otimes \varepsilon:\ \iota^\bullet(\C)\widetilde\otimes  \iota^\bullet(\C)\rightarrow \Omega_{\CO_{N}^{(0)}}^{\bullet}(N)\widetilde\otimes \Omega_{\CO_{\R^0}^{(k)}}^{\bullet}(\BR^{0}).\]
Thus, it follows from Lemmas \ref{prop:tensorhomoequi}, \ref{putong} and \ref{dandian} that the complex map 
 \eqref{eq:commapvarepsilon}
 is a topological homotopy equivalence when $E=\C$.

Recall that, when $E=\C$,  the complex map \eqref{eq:commapeta} is the continuous transpose of the complex map $\varepsilon$. 
Therefore, by  Lemma \ref{lem:trancomthe}, it is a topological homotopy equivalence as well.
\end{proof}

\vspace{3mm}

\noindent\textbf{Proof of Theorem \ref{poin}:}  Proposition \ref{prop:EvaluedDF} implies that the complex map \[
\varepsilon:\iota^\bullet(E)
\rightarrow \Omega_{\CO}^\bullet(M;E)
\] is the same as 
\[
\varepsilon\otimes \mathrm {id}_{\iota^{\bullet}(E)}:\iota^\bullet(\C)\widetilde\otimes \iota^\bullet(E)
\rightarrow \Omega_{\CO}^\bullet(M)\widetilde\otimes \iota^\bullet(E).
\]
Similarly, by the definition of the complex map \eqref{eq:commapeta}, we have that the complex map \[\zeta:\ {\mathrm D}^{-\infty}_c(M;\Omega_{\CO}^{-\bullet})\wt\otimes_{\mathrm{i}}\iota^{\bullet}(E)\rightarrow \iota^\bullet(E)\] coincides with 
\[\zeta\otimes \mathrm {id}_{\iota^{\bullet}(E)}:\ 
{\mathrm D}^{-\infty}_c(M;\Omega_{\CO}^{-\bullet})\widetilde\otimes_{\mathrm{i}}\,\iota^\bullet(E)\rightarrow \iota^\bullet(\C)\widetilde\otimes_{\mathrm{i}}\, \iota^\bullet(E).\]
Then the theorem follows from Lemmas \ref{prop:tensorhomoequi} and \ref{lem:generalpo=C}.
\qed

\vspace{3mm}


\subsection{Poincar\'e's lemma II}\label{sectioncp}
In this subsection, we prove the Poincar\'e's lemma for the de Rham complexes of
$(M,\CO)$ with coefficients in compactly supported formal densities or formal generalized functions.

By taking the restriction of \eqref{eq:zataasamap}, 
we obtain a continuous linear map 
\be \label{eq:map3E=C*}\zeta:\ {\mathrm D}^{\infty}_c(M;\Omega_{\CO}^{0})\rightarrow \BC_{M,\mathrm{f}}(M).\ee 
Then by taking the quasi-completed inductive tensor product with $E$, \eqref{eq:map3E=C*} yields  a continuous linear map 
\be\label{eq:mapetacs*} \zeta:\ {\mathrm D}^{\infty}_c(M;\Omega_{\CO}^{0})\wt\otimes_{\mathrm{i}}E\rightarrow E_{M,\mathrm{f}}(M).\ee
On the other hand,  take the transpose of \eqref{eq:map3E=C*}, and  consider the quasi-completed projective tensor product with $E$, then we obtain a continuous linear map
\be \label{eq:mapvarsc*}
\varepsilon:\ E_M(M)\rightarrow {\mathrm C}^{-\infty}(M;\Omega_{\CO}^0)\wt\otimes E.
\ee

As a slight modification of the de Rham complex \eqref{eq:Dinftyc}, we have a coaugmented de Rham complex
\be\label{eq:ardeRham2} \cdots\xrightarrow{d}{\mathrm D}^{\infty}_c(M;\Omega_{\CO}^{1})\wt\otimes_{\mathrm{i}}E\xrightarrow {d}
{\mathrm D}^{\infty}_c(M;\Omega_{\CO}^{0})\wt\otimes_{\mathrm{i}}E\xrightarrow{\zeta}  E_{M,\mathrm{f}}(M)\rightarrow 0\rightarrow \cdots \ee
  Similarly,  as a slight modification of the de Rham complex \eqref{eq:C-infty},   we have  an augmented de Rham complex
    \be \label{eq:coardeRham2}\cdots\rightarrow 0\rightarrow E_M(M)\xrightarrow{\varepsilon} {\mathrm C}^{-\infty}(M;\Omega_{\CO}^{0})\wt\otimes E \xrightarrow{d} {\mathrm C}^{-\infty}(M;\Omega_{\CO}^{1})\wt\otimes E\xrightarrow{d}\cdots.
\ee

The main goal of this subsection is to prove the following result, which implies Theorem \ref{pfc} with $E=\C$.

\begin{thmd}\label{thm:poinlem2}
	Suppose that $M=(\R^n)^{(k)}$ for some $n,k\in \BN$. Then the complexes \eqref{eq:ardeRham2} and \eqref{eq:coardeRham2}
are both strongly exact.
\end{thmd}

We extend  the maps 
\eqref{eq:mapetacs*} and \eqref{eq:mapvarsc*} to continuous complex maps
\be\label{eq:complexmapetacs} \zeta:\ {\mathrm D}^{\infty}_c(M;\Omega_{\CO}^{-\bullet})\wt\otimes_{\mathrm{i}}\iota^{\bullet}(E)\rightarrow \iota^\bullet(E_{M,\mathrm{f}}(M)),\ee 
and 
\be \label{eq:complexmapvarsc}
\varepsilon:\ \iota^\bullet(E_M(M))\rightarrow {\mathrm C}^{-\infty}(M;\Omega_{\CO}^{\bullet})\wt\otimes \iota^{\bullet} (E),
\ee
respectively.

By Lemmas \ref{lem:tophomimplestrong1} and \ref{lem:tophomimplestrong2}, Theorem \ref{thm:poinlem2} is implied by the following result.  

\begin{thmd}\label{cpoin}
    Suppose that $M=(\R^n)^{(k)}$ for some $n,k\in \BN$. Then the continuous complex maps \eqref{eq:complexmapetacs} and \eqref{eq:complexmapvarsc} are both topological homotopy equivalences.
\end{thmd}

In the rest of this subsection, we prove Theorem \ref{cpoin}.
Suppose that the formal manifold $M$ is as in Theorem \ref{cpoin} in the rest part of this subsection. Note that in this case, we have that 
\[{\mathrm D}^{\infty}_c(M;\Omega_{\CO}^{0})=\RD_c^\infty(\R^n)[y_1^*,y_2^*,\dots,y_k^*]\quad\text{and}\quad \C_{M,\mathrm{f}}(M)=\C\] as LCS by \cite[Proposition 2.17]{CSW2}.
Then by using  \cite[(2.30)]{CSW2} and \eqref{eq:zetaspecial}, the continuous complex map \eqref{eq:complexmapetacs} can be written as 
\[\zeta:\ {\mathrm D}^{\infty}_c(M;\Omega_{\CO}^{-\bullet})\wt\otimes_{\mathrm{i}}\iota^{\bullet}(E)\rightarrow \iota^\bullet(E),\]
with the continuous linear map 
\begin{eqnarray}
\zeta:\ \RD_c^\infty(\R^n)[y_1^*,y_2^*,\dots,y_k^*]\widetilde\otimes_{\mathrm{i}}\, E&\rightarrow &\C\widetilde\otimes_{\mathrm{i}} E=E, \notag\\
\label{eq:zetaofdensity} \left(\sum_{L\in \BN^k} \tau_L (y^*)^L\right)\otimes v&\mapsto & \left(\int_{\R^n} \tau _{(0,0,\dots,0)}\right)\cdot v.
\end{eqnarray}
Similarly, the continuous complex   map  \eqref{eq:complexmapvarsc} can be written as 
\[
\varepsilon:\ \iota^\bullet(E)\rightarrow {\mathrm C}^{-\infty}(M;\Omega_{\CO}^{\bullet})\wt\otimes \iota^{\bullet}(E),
\] with
\begin{eqnarray*}
\varepsilon:\ E_M(M)=E&\rightarrow& {\mathrm C}^{-\infty}(M;\Omega_{\CO}^{0})\wt\otimes E=\CL({\mathrm D}^{\infty}_c(M;\Omega_{\CO}^{0}),E),\\
 v &\mapsto& (\eta\mapsto \zeta(\eta)\cdot v ),
\end{eqnarray*}
where for every $\eta\in {\mathrm D}^{\infty}_c(M;\Omega_{\CO}^{0})$,   $\zeta(\eta)$ denotes the image of $\eta$ under the map \eqref{eq:map3E=C*}.

\begin{lemd}\label{lem:cspoin1} 
The complex map \eqref{eq:complexmapetacs} is a topological homotopy equivalence when $k=0$ and $E=\BC$.
\end{lemd}
\begin{proof}  In this case,
  ${\mathrm D}^{\infty}_c(M;\Omega_{\CO}^{-\bullet})$ agrees with the usual compactly supported
 de Rham complex on $\R^n$ (see Remark \ref{rem:usuPD}). 
Then the lemma is a form of the compactly supported Poincar\'e's lemma for $\R^n$, which is well-known (see  \cite[Corollary 4.7.1]{BT} for example). 
For the convenience of readers,  we present a proof of the lemma below.

Assume first that $n=1$.
We identify ${\mathrm D}^{\infty}_c(\BR;\Omega_{\CO_{\BR}^{(0)}}^{-\bullet})$  with the complex
\[\cdots\rightarrow 0\rightarrow\RC_c^\infty(\R)\xrightarrow{\partial_x}\RC_c^\infty(\R)\rightarrow 0\rightarrow \cdots.\]
Under this identification, by using \eqref{eq:zetaofdensity}, the complex map $\zeta$ is given by  the integration over $\R$.
Fix a smooth function $g\in \RC_c^\infty(\R)$ such that $\int_{-\infty}^{\infty} g(x) dx=1$.
Let $\alpha: \iota^\bullet(\C)\rightarrow  {\mathrm D}^{\infty}_c(\BR;\Omega_{\CO_{\BR}^{(0)}}^{-\bullet})$ be the complex map defined by
\[\alpha:\ \C\rightarrow\RC_c^\infty(\R),\quad \lambda\mapsto \lambda\cdot g.\]
Then we have that $\zeta\circ \alpha=\mathrm{id}_{\iota^{\bullet}(\C)}$.

On the other hand, it is easy to see that the bilinear map
\[\RC_c^\infty(\R)\times \RC_c^\infty(\R)\rightarrow \RC_c^\infty(\R),\quad (f_1,f_2)\mapsto f_1\circledast f_2\]
is well-defined and separately continuous, where
\[(f_1\circledast f_2)(a):=\int_{-\infty}^\infty f_1(x) dx \int_{-\infty}^a f_2(x) dx -\int_{-\infty}^\infty f_2(x) dx\int_{-\infty}^a f_1(x) dx.\]
Let $h:{\mathrm D}^{\infty}_c(\BR;\Omega_{\CO_{\BR}^{(0)}}^{-\bullet})\rightarrow{\mathrm D}^{\infty}_c(\BR;\Omega_{\CO_{\BR}^{(0)}}^{-\bullet})$ be the  continuous map of degree $-1$ determined by
\[h:\ \RC_c^\infty(\R)\rightarrow \RC_c^\infty(\R),\quad f\mapsto f\circledast g.\]
Using the fact that $\int_{-\infty}^{\infty} g(x) dx=1$, one concludes that
\[
\alpha\circ \zeta-\mathrm{id}_{{\mathrm D}^{\infty}_c(\BR;\Omega_{\CO_{\BR}^{(0)}}^{-\bullet})}=d\circ h+h\circ d:\ {\mathrm D}^{\infty}_c(\BR;\Omega_{\CO_{\BR}^{(0)}}^{-\bullet})\rightarrow{\mathrm D}^{\infty}_c(\BR;\Omega_{\CO_{\BR}^{(0)}}^{-\bullet}).
\]
This proves the assertion with $n=1$.

For the general case, Theorem \ref{thm:tensorofcscomplex} implies that
\[{\mathrm D}^{\infty}_c(\BR^n;\Omega_{\CO_{\BR^n}^{(0)}}^{-\bullet})=\underbrace{{\mathrm D}^{\infty}_c(\BR;\Omega_{\CO_{\BR}^{(0)}}^{-\bullet})\widetilde\otimes_{\mathrm{i}}
{\mathrm D}^{\infty}_c(\BR;\Omega_{\CO_{\BR}^{(0)}}^{-\bullet})\widetilde\otimes_{\mathrm{i}}\cdots \widetilde\otimes_{\mathrm{i}} \mathrm D^{\infty}_c(\BR;\Omega_{\CO_{\BR}^{(0)}}^{-\bullet})}_n.\]
Furthermore, Fubini's theorem implies that the complex map \[\zeta:\ \mathrm D^{\infty}_c(\BR^n;\Omega_{\CO_{\BR^n}^{(0)}}^{-\bullet})\rightarrow
\iota^\bullet(\C)\]
coincides with 
\[
\zeta\otimes \cdots \otimes \zeta:
\mathrm D^{\infty}_c(\BR;\Omega_{\CO_{\BR}^{(0)}}^{-\bullet})\widetilde\otimes_{\mathrm{i}}\cdots \widetilde\otimes_{\mathrm{i}} \mathrm D^{\infty}_c(\BR;\Omega_{\CO_{\BR}^{(0)}}^{-\bullet})
\rightarrow
\iota^\bullet(\C)\widetilde\otimes_{\mathrm{i}}\cdots \widetilde\otimes_{\mathrm{i}}\iota^\bullet(\C).
\]
Then the assertion follows from Lemma \ref{prop:tensorhomoequi}.
\end{proof}

\begin{lemd}\label{lem:cspoin2} The complex map \eqref{eq:complexmapetacs} is a  topological homotopy equivalence when  $n=0$ and $E=\C$.
\end{lemd}
\begin{proof} The assertion follows from Lemma \ref{lem:generalpo=C}, as
in this case \eqref{eq:complexmapetacs} is nothing but the  complex map
\eqref{eq:commapeta}.

\end{proof}

\noindent\textbf{Proof of Theorem \ref{cpoin}:}
From  Theorem \ref{thm:tensorofcscomplex}, one  concludes  that
 the complex map \eqref{eq:complexmapetacs}
 coincides with   
 \[\zeta\otimes\zeta\otimes {\mathrm{id}_{\iota^{\bullet}(E)}}:\ \mathrm D^{\infty}_c(\BR^n;\Omega_{\CO_{\BR^n}^{(0)}}^{-\bullet})\widetilde\otimes_{\mathrm{i}} \mathrm D^{\infty}_c(\BR^0;\Omega_{\CO_{\BR^0}^{(k)}}^{-\bullet})
 \widetilde\otimes_{\mathrm{i}} \iota^\bullet(E)\rightarrow \iota^\bullet(\C)\widetilde\otimes_{\mathrm{i}} \iota^\bullet(\C)
 \widetilde\otimes_{\mathrm{i}} \iota^\bullet(E).\]
Thus, by
Lemmas  \ref{lem:cspoin1} and \ref{lem:cspoin2}, it follows form Lemma \ref{prop:tensorhomoequi} that  the complex map
\eqref{eq:complexmapetacs}
is a topological homotopy equivalence.

When $E=\C$, as the transpose of \eqref{eq:complexmapetacs}, the complex map \eqref{eq:complexmapvarsc}
is a topological  homotopy equivalence by Lemma \ref{lem:trancomthe}.
In general, the complex  map \eqref{eq:complexmapvarsc} is also a topological  homotopy equivalence by Lemma \ref{prop:tensorhomoequi}, as it coincides with 
\[
\varepsilon\otimes \mathrm{id}_{\iota^{\bullet}(E)}:\quad \iota^\bullet(\C)\widetilde\otimes \iota^\bullet(E)
\rightarrow\mathrm C^{-\infty}(M;\Omega_{\CO}^{\bullet})\widetilde\otimes \iota^\bullet(E).
\] This finishes the proof.

\appendix

\section{Topological cochain complexes}\label{appendixC}

In this appendix, we collect  some basics on  topological cochain complexes which are needed in 
Sections \ref{sec:derham}-\ref{sec:poin}.

\subsection{Topological cochain complexes}\label{appendixC1}

We start with the following definition. 

\begin{dfn}\label{de:topcomplex}
A topological  cochain complex $I^\bullet$ is a cochain complex
\[\cdots\rightarrow I^i\xrightarrow{d} I^{i+1}\xrightarrow{d} I^{i+2}\rightarrow \cdots\]
of Hausdorff  LCS with continuous linear coboundary maps $d$.
\end{dfn}

Let $I^\bullet$ and $J^\bullet$ be two topological cochain complexes, and let $p\in \Z$. By a continuous map $f:I^\bullet\rightarrow J^\bullet$ of degree $p$, we mean a sequence \be\label{eq:conmap}
f=\{f_n: I^n\rightarrow J^{n+p}\}_{n\in \Z}
\ee
of continuous linear maps. 
As usual, a continuous map of degree $0$ that commutes with the coboundary maps is called a continuous complex map. A continuous complex map $f: I^{\bullet}\rightarrow J^{\bullet}$ is said to be  a topological complex isomorphism if there is a continuous complex map $f': J^{\bullet}\rightarrow I^{\bullet}$ such that $f\circ f'=\mathrm{id}_{J^{\bullet}}$ and $f'\circ f=\mathrm{id}_{I^{\bullet}} $.  Here $\mathrm{id}_{I^\bullet}$ and $\mathrm{id}_{J^\bullet}$ denote the identity maps of complexes.

Let $f, g: I^\bullet\rightarrow J^\bullet$ be two continuous complex maps. 
We say that $f$ is topologically homotopic to $g$, denoted by $f\sim_{\mathrm{h}} g$, if there
is a continuous map $h: I^{\bullet}\rightarrow J^{\bullet}$ of degree $-1$ such that $d\circ h+h\circ d=f-g$. 
The map $f$ is said to be a topological homotopy  equivalence if there is a continuous complex map $f':J^\bullet\rightarrow I^\bullet$
such that $f\circ f'\sim_{\mathrm{h}} \mathrm{id}_{J^\bullet}$ and $f'\circ f\sim_{\mathrm{h}} \mathrm{id}_{I^\bullet}$.

\subsection{Tensor products of topological cochain complexes}\label{appendixC2}
Let $E$ and $F$ be two  LCS. 
As discussed in \cite{Gr}, the algebraic tensor product $E\otimes F$ admits at least three valuable locally convex topologies:  the inductive tensor product $E\otimes_{\mathrm{i}} F$, the projective tensor product $E\otimes_{\pi} F$, and the epsilon (injective) tensor product $E\otimes_{\varepsilon} F$. Explicit definitions can be found in \cite{Gr}.
In this context, as mentioned  in the Introduction,  we denote  the quasi-completions and completions of these topological tensor product spaces as follows:
\[\text{$E\widetilde\otimes_{\mathrm{i}} F$,
$E\widetilde\otimes_{\pi} F$, $E\widetilde\otimes_{\varepsilon} F$\quad and\quad $E\widehat\otimes_{\mathrm{i}} F$,
$E\widehat\otimes_{\pi} F$, $E\widehat\otimes_{\varepsilon} F$}.\]
When $E$ or $F$ is nuclear, it is a classical result of Grothendieck that their projective tensor product coincides with the epsilon tensor product (see \cite[Theorem 50.1]{Tr}).
 In this case, we
will simply write
\bee
E\widetilde\otimes F:=E\widetilde\otimes_{\pi} F=E\widetilde\otimes_{\varepsilon} F
\quad
\text{and}
\quad
E\widehat\otimes F:=E\widehat\otimes_{\pi} F=E\widehat\otimes_{\varepsilon} F.
\eee
Given two continuous linear maps $\phi_1:E_1\rightarrow F_1$ and $\phi_2:E_2\rightarrow F_2$  between  LCS, we use $\phi_1\otimes \phi_2$ to denote the continuous linear map on  various topological tensor products (or their quasi-completions or completions) obtained by the tensor product of $\phi_1$ and $\phi_2$. For example, we have the maps \[\phi_1\otimes \phi_2:\ E_1\widehat\otimes_{\pi} E_2\rightarrow F_1\widehat\otimes_{\pi} F_2\quad \text{and}\quad 
\phi_1\otimes \phi_2:\ E_1\widetilde\otimes_{\mathrm{i}} E_2\rightarrow F_1\widetilde\otimes_{\mathrm{i}} F_2.\]

Let  $I^\bullet$ and $J^\bullet$ be two topological cochain complexes. 
We define
$I^\bullet\widetilde\otimes_\Box J^\bullet$  and $I^\bullet\widehat\otimes_\Box J^\bullet$ to be the topological cochain complexes such that
\[
  (I^\bullet\widetilde\otimes_\Box J^\bullet)^i=\bigoplus_{r,s\in \Z; r+s=i} I^r\widetilde \otimes_\Box J^s,\quad
 (I^\bullet\widehat\otimes_\Box J^\bullet)^i=\bigoplus_{r,s\in \Z; r+s=i} I^r\widehat \otimes_\Box J^s
\]
for all $i\in \Z$, and 
\[
   d(u\otimes v)=(du)\otimes v+ (-1)^r u\otimes (d v)
   \]
   for all $r,s\in \Z$, $u\in I^r$ and $v\in J^s$. Here the notation $\otimes_\Box$ stands for one of the topological tensor products $\otimes_\pi$, $\otimes_\varepsilon$ and $\otimes_{\mathrm{i}}$. Then both $I^\bullet\widetilde\otimes_\Box J^\bullet$ and $I^\bullet\widehat\otimes_\Box J^\bullet$ are 
   topological cochain complexes.


 We say that a topological cochain complex $I^\bullet$ is nuclear if the LCS $I^i$ is nuclear for all $i\in \Z$.
 If  $I^\bullet$ or $J^\bullet$ is nuclear, we
   will simply write
\bee
I^\bullet \widetilde\otimes J^\bullet:= I^\bullet\widetilde\otimes_{\pi} J^\bullet=I^\bullet \widetilde\otimes_{\varepsilon} J^\bullet
\quad
\text{and}
\quad
I^\bullet \widehat\otimes J^\bullet:=I^\bullet\widehat\otimes_{\pi} J^\bullet=I^\bullet\widehat\otimes_{\varepsilon} J^\bullet.
\eee

\begin{lemd}\label{f1} Let $I^\bullet$, $J^\bullet$ and $K^\bullet$  be  
topological cochain complexes, and
   let $f_1, f_2: I^\bullet\rightarrow K^\bullet$ be two continuous complex maps. If $f_1\sim_{\mathrm{h}}f_2$,
  then  \[	f_1\otimes  \mathrm{id}_{J^{\bullet}} \sim_{\mathrm{h}}f_2\otimes  \mathrm{id}_{J^{\bullet}} :\quad I^\bullet\widetilde\otimes_\Box  J^\bullet\rightarrow K^\bullet\widetilde\otimes_\Box  J^\bullet,\] and  \[	f_1\otimes  \mathrm{id}_{J^{\bullet}} \sim_{\mathrm{h}}f_2\otimes  \mathrm{id}_{J^{\bullet}} :\quad I^\bullet\wh\otimes_\Box  J^\bullet\rightarrow K^\bullet\wh\otimes_\Box  J^\bullet.\]
\end{lemd}
\begin{proof}
  Choose a continuous map $h:I^\bullet\rightarrow K^\bullet$  of degree $-1$ such that  $d\circ h+h\circ d=f_1-f_2.$
  Then
  \[h\otimes \mathrm{id}_{J^{\bullet}} :I^\bullet\widetilde\otimes_\Box J^\bullet\rightarrow K^\bullet\widetilde\otimes_\Box J^\bullet
  \quad (\text{resp.}\ h\otimes  \mathrm{id}_{J^{\bullet}} :I^\bullet\widehat\otimes J^\bullet\rightarrow K^\bullet\widehat\otimes J^\bullet)\]
  is a continuous map of degree $-1$ such that $d\circ (h\otimes  \mathrm{id}_{J^{\bullet}}) +(h\otimes  \mathrm{id}_{J^{\bullet}}) \circ d$ and  $f_1\otimes  \mathrm{id}_{J^{\bullet}}- f_2\otimes  \mathrm{id}_{J^{\bullet}}$
  agree on the strictly dense (resp.\,dense) subspace $I^r\otimes J^s$ of $I^r\widetilde\otimes_\Box J^s$ (resp.\,$I^r\widehat\otimes_\Box J^s$), where $r,s\in \Z$. 
 The lemma then follows.
\end{proof}

\begin{lemd} \label{prop:tensorhomoequi}
Let $I^\bullet$, $J^\bullet$, $I_1^\bullet$ and $J_1^\bullet$  be  
topological cochain complexes.
Suppose that $f: I^\bullet\rightarrow I_1^\bullet$ and $g:  J^\bullet\rightarrow J_1^\bullet$ are topological homotopy equivalences. Then the
continuous complex maps
\be\label{eq:tenprocommap}
  f\otimes g: I^\bullet\widetilde\otimes_\Box J^\bullet\rightarrow I_1^\bullet\widetilde\otimes_\Box J_1^\bullet\quad\text{and}\quad
  f\otimes g: I^\bullet\widehat\otimes_\Box J^\bullet\rightarrow I_1^\bullet\widehat\otimes_\Box J_1^\bullet
\ee
are topological homotopy equivalences as well.

\end{lemd}
\begin{proof}
This easily follows from Lemma \ref{f1}.
\end{proof}

\subsection{Transposes of topological cochain complexes}\label{appendixC3}

Let $E$ and $F$ be two LCS. As mentioned in the Introduction, let  $E'$  denote the space of all continuous linear maps from $E$ to $\BC$, which is endowed with the strong topology. For a continuous linear map $\phi :E\rightarrow F$, let\be \label{eq:contrans} {}^t \phi:\ F' \rightarrow E',\quad u'\mapsto (v\mapsto \la u', \phi(v)\ra)\ee denote the transpose of $\phi$. It is well-known that ${}^t \phi$ is a continuous linear map.

Let $I^\bullet$ be a topological cochain complex. Write $({}^{t}I)^{\bullet}$ for
the topological cochain complex, to be called the transpose of  $I^\bullet$,
such that for all $i\in \Z$,
\[({}^{t}I)^{i}:=(I^{-i})'\]
 and 
\[\la d(u'),v\ra=(-1)^{i}\la u',d(v)\ra\quad (u'\in ({}^{t}I)^{i}, v\in I^{-i-1}).\]

Let $f:I^\bullet\rightarrow J^\bullet$ be a continuous complex map.
As in \eqref{eq:contrans}, we define the transpose   ${}^tf:({}^tJ)^\bullet\rightarrow ({}^tI)^\bullet$ of $f$ in an obvious way (by taking the term-wise transpose), which
is still a continuous complex map. Furthermore, we have the following result.

\begin{lemd}\label{lem:trancomthe} Let $f,g:I^\bullet\rightarrow J^\bullet$ be two continuous complex maps which are topologically homotopic  to each other. 
	Then
\[{}^tf,\ {}^tg:\quad ({}^tJ)^\bullet\rightarrow ({}^tI)^\bullet\] are topologically homotopic  to each other as well. 
 Moreover, if $f:I^\bullet\rightarrow J^\bullet$ is a topological homotopy equivalence,
then so is ${}^tf:({}^tJ)^\bullet\rightarrow ({}^tI)^\bullet$.
\end{lemd}
\begin{proof}Let  $h:I^\bullet\rightarrow J^\bullet$  be a continuous map of degree $-1$ such that  $d\circ h+h\circ d=f-g.$
Then
\[{}^{\#}h:\ ({}^tJ)^\bullet\rightarrow ({}^tI)^\bullet,\quad u'\mapsto (v\mapsto (-1)^i \la u',h(v)\ra)\quad
(u'\in ({}^{t}J)^{i+1}, v\in I^{-i}, i\in \Z)\]
is a continuous map of degree $-1$ such that  $d\circ {}^{\#}h+{}^{\#}h\circ d={}^tf-{}^tg$.
This proves the first assertion. The second assertion easily follows from the first one.
\end{proof}

\subsection{Strong exactness of topological cochain complexes}\label{appendixC4}
Recall from the Introduction that a topological cochain complex is said to be strongly exact if it is exact and all the coboundary maps are strong.
For every  Hausdorff LCS $E$, form the topological cochain complex concentrated at degree $0$: 
\be\label{eq:iotaE} \iota^{\bullet}(E):\quad \cdots\rightarrow  0\rightarrow 0\rightarrow \iota^0(E):=E\rightarrow 0\rightarrow 0\rightarrow \cdots.\ee

We say that a topological cochain complex $I^\bullet$ is positive (resp.\,negative) if $I^i=0$ for all $i<0$ (resp.\,$i>0$).
The following two results are useful in checking the strong exactness of positive/negative topological cochain complexes.

\begin{lemd}\label{lem:tophomimplestrong1}
Let $I^\bullet$ be a positive topological cochain complex, and let $E$ be
a Hausdorff LCS. Suppose that
 there is a topological homotopy equivalence  $f:\iota^\bullet(E)\rightarrow I^\bullet$.
Then the induced topological cochain complex
\be\label{eq:complexI+E} \cdots\rightarrow 0\rightarrow 0\rightarrow E\xrightarrow{f} I^{0} \xrightarrow{d} I^{1}\xrightarrow{d} I^{2}\rightarrow \cdots \ee
is strongly exact.
\end{lemd}
\begin{proof} The exactness of \eqref{eq:complexI+E} is obvious, and we need to prove that the continuous linear maps
 $f:E\rightarrow I^0$ and $d:I^r\rightarrow I^{r+1}$ ($r\ge 0$) are strong.
Let $g:I^\bullet\rightarrow \iota^\bullet(E)$ be a continuous complex map  such that $f\circ g\sim_{\mathrm{h}} \mathrm{id}_{I^\bullet}$ and $g\circ f\sim_{\mathrm{h}} \mathrm{id}_{\iota^\bullet(E)}$.
Since $g\circ f=\mathrm{id}_{\iota^\bullet(E)}$, we have that  $f=f\circ g\circ f:E\rightarrow I^0$.	

We take  a continuous map $h:I^\bullet\rightarrow I^\bullet$ of degree $-1$ such that
$d\circ h+h\circ d=\mathrm{id}_{I^\bullet}-f\circ g$.
Since $g\circ d=0$ on $I^r$, we find that
 \[d\circ h\circ d=(\mathrm{id}_{I^{r+1}}-f\circ g-h\circ d)\circ d=d\]
 as continuous linear maps from $I^r$ to $I^{r+1}$. This completes the proof.
\end{proof}

Similar to Lemma \ref{lem:tophomimplestrong1}, we also have the following result.

\begin{lemd}\label{lem:tophomimplestrong2}
Let $J^\bullet$ be a negative topological cochain complex, and let $E$ be
a Hausdorff LCS. Suppose that
there is a topological homotopy equivalence $g:J^\bullet\rightarrow \iota^\bullet(E)$.
Then the topological cochain complex
\be \cdots\rightarrow  J^{-2} \xrightarrow{d} J^{-1}\xrightarrow{d} J^{0}\xrightarrow{g}E\rightarrow 0 \rightarrow 0\rightarrow \cdots \ee
is strongly exact.
\end{lemd}

\section*{Acknowledgement}
F. Chen is supported by the National Natural Science Foundation of China (Nos. 12131018, 12161141001) and the Fundamental Research Funds for the Central Universities (No. 20720230020).
B. Sun is supported by  National Key R \& D Program of China (Nos. 2022YFA1005300 and 2020YFA0712600) and New Cornerstone Investigator Program. 
The first and the third authors would like to thank Institute for Advanced Study in Mathematics, Zhejiang University. Part of this work was carried out while they were visiting the institute.

\end{document}